\documentclass[11pt]{article}
\usepackage[a4paper,margin=1in]{geometry}
\usepackage{graphicx}
\usepackage{epstopdf}
\usepackage{amssymb}
\usepackage{mathtools}
\usepackage{bm}
\usepackage{amsmath}
\usepackage{dsfont}
\usepackage{extarrows}
\usepackage{caption}
\usepackage{float}
\usepackage{verbatim}
\usepackage{epsfig}
\usepackage{multirow}
\usepackage{subfigure}
\usepackage{booktabs}
\usepackage{array}
\usepackage{threeparttable}
\usepackage{enumitem}
\usepackage{algorithm}
\usepackage{algorithmic}
\usepackage{hyperref}
\usepackage{xcolor}
\usepackage{colortbl}
\usepackage{amsthm}

\nonstopmode
\makeatletter

\newcommand{\Rmnum}[1]{\expandafter\@slowromancap\romannumeral #1@}
\makeatother

\usepackage{cleveref}
\usepackage{microtype}
\usepackage{needspace}
\crefname{algorithm}{Algorithm}{Algorithms}
\crefname{table}{Table}{Tables}
\usepackage[section]{placeins}

\numberwithin{equation}{section}%
\crefname{lemma}{Lemma}{Lemma}
\crefname{theorem}{Theorem}{Theorem}
\crefname{definition}{Definition}{Definition}
\crefname{remark}{Remark}{Remark}

\crefname{example}{Example}{Examples}
\crefname{table}{Table}{Tables}
\crefname{figure}{Figure}{Figures}
\newtheorem{theorem}{Theorem}[section]
\newtheorem{lemma}[theorem]{Lemma}
\theoremstyle{definition}
\newtheorem{example}[theorem]{Example}
\theoremstyle{remark}
\newtheorem{remark}[theorem]{Remark}
\newcommand{\keywords}[1]{\par\medskip\noindent\textbf{Keywords:} #1}
\newcommand{\subclass}[1]{\par\smallskip\noindent\textbf{2020 Mathematics Subject Classification:} #1}
\begin{document}



\title{A Weighted Integral-Regularized Finite Difference Scheme for the Tempered Fractional Laplacian\\
\large %
\thanks{The research of Dongling Wang is supported in part by the National Natural Science Foundation of China under grants 12271463,
and the Natural Science Foundation of Hunan Province under Grant 2026JJ50363.
The work of Mingyi Wang is supported by Postgraduate Scientific Research Innovation Project of Hunan Province (No. CX20240602). \\ Declarations of interest: none.}
}

\author{Mingyi Wang \qquad Lisen Ding \qquad Dongling Wang\thanks{Corresponding author.}\\[0.5em]
\small Hunan Research Center of the Basic Discipline Fundamental Algorithmic Theory\\
\small and Novel Computational Methods, National Center for Applied Mathematics in Hunan,\\
\small School of Mathematics and Computational Science, Xiangtan University,\\
\small Xiangtan 411105, Hunan, China\\[0.5em]
\small \texttt{202331510144@smail.xtu.edu.cn} (Mingyi Wang)\\
\small \texttt{dingmath15@smail.xtu.edu.cn} (Lisen Ding)\\
\small \texttt{wdymath@xtu.edu.cn} (Dongling Wang)\\
\small ORCID: 0000-0001-8509-2837 (Dongling Wang)}

\date{}

\maketitle



\begin{abstract}
The intrinsic singularity of the tempered fractional Laplacian (TFL) remains a major challenge in developing numerical methods that are simultaneously accurate, efficient, and easy to implement. We develop a weighted integral-regularized finite difference (WIRFD) method that regularizes the singular integrand via a multidimensional Taylor expansion incorporating a smooth window function. The resulting integral is decomposed into a regularized term, which is discretized by a punctured trapezoidal rule, and a directly evaluated correction term. For the multidimensional TFL operator, we derive an $O(h^{4-\alpha})$ truncation error bound in the $l^{\infty}$-norm for $\alpha\in(0,2)$ and $u\in C^s(\mathbb{R}^d)$ with $s\geq 8$ by introducing a smooth auxiliary function together with the aliasing formula. For the one-dimensional TFL equation, we establish stability in both the $l^2$- and $l^{\infty}$-norms and optimal $O(h^{4-\alpha})$ convergence for $\alpha\in[1,2)$ based on the strict diagonal dominance of the discrete matrix and a lower bound for its minimum eigenvalue. The Toeplitz structure of the discrete matrix enables FFT-based matrix-vector multiplication, and the resulting linear systems are solved efficiently by a preconditioned conjugate gradient (PCG) method. Numerical experiments corroborate the theoretical results, demonstrating the accuracy, efficiency, and robustness of the proposed method.
\subclass{65N06 \and 35R11 \and 65N12 \and 65N15}
\end{abstract}



\section{Introduction}



This work presents a weighted integral-regularized finite difference (WIRFD) method for the tempered fractional Laplacian (TFL) equation subject to the homogeneous exterior Dirichlet condition \cite{DZ2019JSC}:
\begin{alignat}{2}\label{eq:1}
(-\Delta)_{\lambda}^{\frac{\alpha}{2}}u(x) &= f(x), \quad && x \in \Omega \subset \mathbb{R}^d, \nonumber \\
u(x) &= 0, \quad && x \in \Omega^c:=\mathbb{R}^d \setminus \Omega,
\end{alignat}
with fractional order $\alpha\in(0,2)$, tempering factor $\lambda\geq 0$, spatial dimension $d\in\{1,2,3\}$, and source function $f(x)$. The TFL $(-\Delta)_{\lambda}^{\frac{\alpha}{2}}$ takes the form of a Cauchy principal value (P.V.) integral \cite{DLTZ2018,DZ2019JSC,ZDF2019}:
\begin{align}\label{eq:2}\displaystyle
(-\Delta)_{\lambda}^{\frac{\alpha}{2}}u(x):=C_{d,\alpha}\text{P.V.}\int_{\mathbb{R}^d}{\frac{u(x)-u(y)}{e^{\lambda|x-y|}|x-y|^{d+\alpha}}}dy,
\end{align}
where the normalization constant $C_{d,\alpha}$ is given explicitly in terms of the Gamma function $\Gamma(\cdot)$:
\begin{align}\label{eq:3}
C_{d,\alpha}=\begin{cases}\displaystyle\frac{\alpha\Gamma\left(\frac{d+\alpha}{2}\right)}{2^{1-\alpha}\pi^{\frac{d}{2}}\Gamma\left(1-\frac{\alpha}{2}\right)},\quad &\lambda=0\  \text{or}\ \alpha=1,\\
\displaystyle \frac{\Gamma\left(\frac{d}{2}\right)}{2\pi^{\frac{d}{2}}|\Gamma(-\alpha)|},\quad &\lambda>0,\ \alpha\ne 1.\end{cases}
\end{align}
The tempered fractional Laplacian (TFL) is a nonlocal operator that preserves the singular kernel structure of the fractional Laplacian while incorporating finite-scale effects through exponential tempering.
For $\lambda=0$, the TFL $(-\Delta)_{\lambda}^{\frac{\alpha}{2}}$ recovers the fractional Laplacian $(-\Delta)^{\frac{\alpha}{2}}$. For $\lambda>0$, the kernel decays exponentially at infinity, so that the operator captures nonlocal effects through tempered long-range interactions. This suppresses the heavy-tailed behavior of the pure power-law kernel, providing a flexible mathematical framework for modeling finite-scale systems and bridging idealized heavy-tailed models and finite-scale real-world applications. When the TFL is restricted to a bounded domain, the homogeneous exterior Dirichlet condition provides the missing exterior information required in the definition of the nonlocal operator. It also allows the problem to be formulated in a variational framework within which coercivity, self-adjointness, spectral discreteness, and well-posedness can be established.


The TFL $(-\Delta)_{\lambda}^{\frac{\alpha}{2}}$ serves as the infinitesimal generator of a tempered L\'evy process \cite{DLTZ2018}, characterizing the transition from anomalous diffusion to normal diffusion \cite{BM2010,MZB2008,Jan2007}. The TFL equation finds widespread application across diverse domains, encompassing finance \cite{CGMY2002,SMC2015}, physics \cite{R2000}, biology \cite{JHMJM2013}, and geophysics \cite{MZB2008,ZMP2012}. Owing to the general lack of a closed-form solution, the development of concise and efficient numerical methods is essential. However, the inherent nonlocality, kernel singularity, and multidimensional nature of this class of equations pose substantial challenges for constructing effective numerical methods. For $\lambda=0$, a mature numerical framework is well established, primarily comprising finite difference methods \cite{DvZ2018,DZ2019,HO2014,HO2016,MY2020,HZD2021,HW2022,YCYL2025}, finite element methods \cite{AB2017,ABB2017,DG2013,SL2024}, and spectral methods \cite{MS2017,TWYZ2020,ZZ2024}. For $\lambda>0$, by contrast, available numerical methods for the TFL equation remain limited, particularly those that simultaneously offer high-order accuracy, rigorous stability and convergence analysis, and efficient implementation in multiple dimensions. This gap motivates the development of robust and efficient numerical methods for the TFL equation.

In recent years, considerable effort has been devoted to the numerical approximation of the TFL equation, with particular attention to the nonlocal coupling and the hypersingular kernel. Among the earliest studies on numerical methods for the one-dimensional TFL equation, Zhang et al. \cite{ZDK2018} developed a Riesz basis finite element method, providing rigorous proofs for the well-posedness of its Galerkin weak formulation and the convergence of the proposed method. Thereafter, they devised a finite difference scheme, demonstrating its convergence and employing tailored preconditioners to solve the resulting linear system \cite{ZDF2019}. Aiming for higher accuracy, Duo and Zhang \cite{DZ2019JSC} proposed a second-order finite difference scheme, systematically investigating the tempering effects on the solution. In two dimensions, Sun et al. \cite{SND2021} formulated a finite difference scheme based on the weighted trapezoidal rule and bilinear interpolation, deriving the corresponding error estimates. Meanwhile, Yan et al. \cite{YDN2020} presented a finite difference scheme using linear and bilinear interpolation for the nonsingular components and quadratic and biquadratic interpolation for the singular components, proving that both the truncation and numerical errors are of order $O(h^{2-\alpha})$. In a more recent development, Cui et al. \cite{CDH2025} contributed an efficient and accurate second-order finite difference scheme built on a generating function, delivering a detailed stability and convergence analysis. Pursuing higher accuracy, Wang et al. \cite{WW2026} constructed higher-order finite difference methods using new generating functions through discrete symbols, establishing stability and convergence of orders $p=4,6,8$.


The methods reviewed above have made substantial advances in accuracy, singularity treatment, and dimensionality. However, a direct real-space regularization framework that extends singularity subtraction to the TFL, retains high-order accuracy for all $\alpha\in(0,2)$, and admits a rigorous multidimensional truncation error analysis remains less developed. This motivates us to propose a concise and efficient weighted integral-regularized finite difference (WIRFD) method for the TFL equation. The principal innovations and contributions of this work are as follows.

\begin{itemize}

  \item \textbf{A high-order WIRFD scheme for the TFL.} We propose a WIRFD scheme for the TFL based on an integral regularization strategy combining a multidimensional Taylor expansion, a smooth window function, and an addition-subtraction technique. The resulting regularized term is discretized by a punctured trapezoidal rule, while the correction term is directly evaluated, yielding a high-order scheme.
      By introducing a smooth auxiliary function and employing the aliasing formula, we prove an $O(h^{4-\alpha})$ truncation error estimate in the $l^\infty$-norm for $\alpha\in(0,2)$, assuming $u\in C^s(\mathbb{R}^d)$ ($s\ge8$) and a suitable window function $w \in C^\gamma([0,\infty))$ ($\gamma\geq4$). This analysis extends the singularity subtraction framework of \cite{MY2020} from the fractional Laplacian to the TFL and provides a multidimensional truncation error analysis for the proposed scheme.

  \item \textbf{Stability and optimal convergence for the $d=1$ TFL equation.} For the $d=1$ TFL equation with $\alpha\in[1,2)$, we establish stability and optimal $O(h^{4-\alpha})$ convergence in the $l^2$- and $l^\infty$-norms. The proof relies on a systematic analysis of the discrete matrix, including strict diagonal dominance and a lower bound for its minimum eigenvalue.
\end{itemize}


The remainder of this article is organized as follows. Section 2 devises a WIRFD scheme for the multidimensional TFL via an integral
regularization strategy. Section 3 rigorously establishes an $O(h^{4-\alpha})$ truncation error bound for the proposed scheme in the $l^\infty$-norm, presenting the corresponding algorithmic implementation. Section 4 demonstrates the stability and $O(h^{4-\alpha})$ convergence of the WIRFD method for the $d=1$ TFL equation in both the $l^2$- and $l^\infty$-norms, detailing the algorithmic procedure incorporating a preconditioner. Section 5 numerically examines the effects of the exponential factor $\lambda$ on the error and convergence rate of the WIRFD scheme for the TFL and validates its performance through systematic numerical experiments. Section 6 concludes the paper and outlines prospects for future research.



\section{Finite difference discretization of the multidimensional TFL}

When $\alpha>1$, the TFL defined in \eqref{eq:2} becomes hypersingular as $x \to y$, rendering direct quadrature unreliable. To circumvent such numerical inaccuracy, this section constructs a weighted integral-regularized finite difference (WIRFD) scheme for the TFL in $d=1,2,3$ via singularity subtraction, which regularizes the singular integrand and allows the resulting integral to be accurately discretized by the punctured trapezoidal rule.

\subsection{Singular integrand regularization}

For $u\in C^{n+1}(\mathbb{R}^d)$ and $x\in\mathbb{R}^d$, the Taylor expansion of $u$ about $x$, evaluated at $x+y$, takes the form
\begin{align}\label{eq:2.1}
u(x+y) = u(x)+\sum_{1\leq|k|\leq n}\frac{D^k u(x)}{k!}y^k+R_n(y),\quad y\to 0,
\end{align}
where $k=(k_1,\ldots,k_d)\in\mathbb{N}_0^d$ is a multi-index, $\mathbb{N}_0=\{0,1,2,\ldots\}$, and
$|k|=\sum_{l=1}^d k_l,\, k!=\prod_{l=1}^d{k_l}!,\, D^k =\frac{\partial^{|k|}}{\partial x_1^{k_1} \ldots \partial x_d^{k_d}},\, y^k =\prod_{l=1}^d (y_l)^{k_l}.$
The integral remainder is expressed as
\begin{align}\label{eq:2.2}
R_n(y) = \sum_{|\beta|=n+1}\frac{n+1}{\beta!}y^{\beta}\int_{0}^1(1-t)^{n}D^{\beta}u(x+ty)\,dt.
\end{align}
Since $u\in C^{n+1}(\mathbb{R}^d)$, each $D^\beta u$ with $|\beta|=n+1$ is bounded in a neighborhood of $x$. Consequently, $R_n(y)=O(|y|^{n+1})$ as $y\to 0$.


In what follows, the singularity subtraction technique based on the Taylor expansion is employed to regularize the singular integrand. The numerator factor of the resulting regularized integrand is then rigorously proved to belong to $C^q(\mathbb{R}^d)$ for some $q\in \mathbb{N}_0$.


For singularity subtraction, we introduce a compactly supported window function $w(\rho)$, with $\rho=|y|$, satisfying
\begin{align}\label{eq:2.3}
w(\rho)=1+O(\rho^\gamma),\quad \rho\to 0,\quad \gamma\in \mathbb{N}_0.
\end{align}
For $\gamma\geq2$, with $\operatorname{supp}w\subset[0,L]$ and $w \neq 0$ on $[0,L)$, the condition $w\in C^\gamma([0,\infty))$ implies $w^{(\ell)}(0)=0$ for $1\leq \ell\leq\gamma-1$ and $w^{(\ell)}(L)=0$ for $0\leq \ell\leq\gamma$, a fact used below in the local regularity analysis.

After a change of variables, the integral in \eqref{eq:2} is regularized by adding and subtracting, in the numerator of the integrand, the product of the window function $w(|y|)$ and the third-order Taylor polynomial of $u(x+y)-u(x)$,

\begin{align}\label{eq:2.4}
(-\Delta)_{\lambda}^{\frac{\alpha}{2}}u(x)
=
&C_{d,\alpha}\int_{\mathbb{R}^d}{\frac{ u(x)-u(x+y)+w(|y|)\sum_{1\leq|k|\leq 3}\frac{D^k u(x)}{k!}y^k}{e^{\lambda|y|}|y|^{d+\alpha}}}dy\nonumber\\
&-C_{d,\alpha}\int_{\mathbb{R}^d}{\frac{ w(|y|)\sum_{1\leq|k|\leq 3}\frac{D^k u(x)}{k!}y^k}{e^{\lambda|y|}|y|^{d+\alpha}}}dy\nonumber\\
:=&J_1+J_2,
\end{align}
where $J_1$ and $J_2$ denote the first and second terms on the right-hand side, namely the regularized term and the corresponding correction term, respectively.

We now show that \eqref{eq:2.4} removes the singularity in \eqref{eq:2} by proving that the numerator of the integrand for $J_1$ belongs to $C^q(\mathbb{R}^d)$ and that the integrands of $J_1$ and $J_2$ are integrable in $y$.


\begin{lemma}\label{L2.1}
Let $u\in C^s(\mathbb{R}^d)$ be compactly supported, with $s\geq 4$, and let $w\in C^\gamma([0,\infty))$ be as in \eqref{eq:2.3}, with $w(\rho)=0$ for $\rho\geq L>0$. Set
\begin{align}\label{eq:2.5}
v(y):=u(x)-u(x+y)+w(|y|)\sum_{1\leq|k|\leq 3}\frac{D^k u(x)}{k!}y^k.
\end{align}
Then the following properties hold.
\begin{enumerate}
\item[\textup{(i)}] $v\in C^q(\mathbb{R}^d)$ with $q=\min\{s-4,\gamma\}$, where $\gamma\in \mathbb{N}_0$ is as in \eqref{eq:2.3}. Moreover, for any multi-index $k$ with $0\leq |k|\leq \min\{\tau,q\}$,
    \begin{align}
    D^kv(y)=O\left(|y|^{\tau-|k|}\right),\qquad y\to 0,
    \end{align}
    where $\tau=\min\{4,\gamma+1\}$.
\item[\textup{(ii)}] For $\gamma\in \mathbb{N}_0$, the integrand of $J_1$ in \eqref{eq:2.4} is absolutely integrable over $\mathbb{R}^d$.
\item[\textup{(iii)}] $J_2$ in \eqref{eq:2.4} is given explicitly by
    \begin{align}\label{m2.7}
    J_2=-\frac{C_{d,\alpha}\Delta u(x)}{2d}\int_{[-L,L]^d}{\frac{w(|y|)}{e^{\lambda|y|}|y|^{d+\alpha-2}}}dy= - \frac{C_{d,\alpha}\pi^{\frac{d}{2}}\Delta u(x)}{d\Gamma\left(\frac{d}{2}\right)}\int_{0}^{L}\frac{w(\rho)}{e^{\lambda \rho}\rho^{\alpha-1}}d\rho,
    \end{align}
  where $\Gamma(\cdot)$ denotes the Gamma function.
\end{enumerate}
\end{lemma}


\begin{proof}
Part (i). For $y\in Q_\delta:=[-\delta,\delta]^d$ with $\delta>0$ sufficiently small, substituting the expansion of $u(x)-u(x+y)$ from \eqref{eq:2.1} into \eqref{eq:2.5} gives
\begin{align}\label{eq:2.8}
v(y) = (w(|y|)-1)\sum_{1\leq|k|\leq 3}\frac{D^k u(x)}{k!}y^k-R_3(y).
\end{align}
\eqref{eq:2.8} relates the local regularity of $v$ in $Q_{\delta}$ to that of $w$ and $R_3$. The vanishing derivatives of $w-1$ at the origin are the key point in verifying the Cartesian regularity of the products $(w(|y|)-1)y^k$. In light of \eqref{eq:2.2} and $u\in C^s(\mathbb{R}^d)$, we have $R_3\in C^{s-4}(Q_\delta)$. Combining this with $w\in C^\gamma([0,\infty))$, we obtain
\begin{align}
v\in C^q(Q_\delta), \quad q = \min\{s-4,\gamma\}.
\end{align}
Under the same assumptions, \eqref{eq:2.5} also implies that
\begin{align}
v\in C^{\min\{s-3,\gamma\}}(\mathbb{R}^d\setminus Q_\delta).
\end{align}
Hence, combining the interior and exterior regularity estimates, we conclude that $v \in C^q(\mathbb{R}^d)$.

Applying \eqref{eq:2.2} and \eqref{eq:2.3} to \eqref{eq:2.8}, we deduce that
\begin{align}\label{eq:02.9}
v(y)=O(|y|^{\tau}),\quad y\to 0,
\end{align}
where $\tau=\min\{4,\gamma+1\}$.
From \eqref{eq:02.9} and $v\in C^{q}(Q_\delta)$, differentiating \eqref{eq:2.8} yields, for $0\leq |k|\leq \min\{\tau,q\}$,
\begin{align}
D^kv(y)=O\left(|y|^{\tau-|k|}\right), \quad y\to 0.
\end{align}
Part (ii). By \eqref{eq:2.5}, $J_1$ in \eqref{eq:2.4} can be decomposed as
\begin{align}\label{eq:2.9}
J_1&=C_{d,\alpha}\int_{\mathbb{R}^d}\frac{v(y)}{e^{\lambda|y|}|y|^{d+\alpha}}dy\nonumber\\
&=C_{d,\alpha}\left(\int_{Q_\delta}\frac{v(y)}{e^{\lambda|y|}|y|^{d+\alpha}}dy+\int_{Q_{R_0}\setminus Q_{\delta}}\frac{v(y)}{e^{\lambda|y|}|y|^{d+\alpha}}dy+\int_{\mathbb{R}^d\setminus Q_{R_0}}\frac{v(y)}{e^{\lambda|y|}|y|^{d+\alpha}}dy\right)\nonumber\\
&:= I_1+I_2+I_3,
\end{align}
where $Q_{R_0}=[-R_0,R_0]^d$, and $I_1$, $I_2$, $I_3$ are the three right-hand side terms, respectively.

For $\gamma\in \mathbb{N}_0$, applying \eqref{eq:2.3} and \eqref{eq:2.8} to $I_1$ yields
\begin{align}\label{eq:2.11}
I_1 = C_{d,\alpha}\int_{Q_\delta}\frac{O(|y|^{\gamma})\sum_{1\leq|k|\leq 3}\frac{D^k u(x)}{k!}y^k-R_3(y)}{e^{\lambda|y|}|y|^{d+\alpha}}dy.
\end{align}
Using \eqref{eq:2.2}, the symmetry of $Q_\delta$, and a polar coordinate transformation, we obtain from \eqref{eq:2.11} the following estimate:
\begin{align}\label{eq:2.12}
I_1 &= C_{d,\alpha}\int_{Q_\delta}\frac{O(|y|^\gamma)\sum_{|k|=2}\frac{D^k u(x)}{k!}y^k-R_3(y)}{e^{\lambda|y|}|y|^{d+\alpha}}dy\nonumber\\
&=C_{d,\alpha}\int_{Q_\delta}\frac{O(|y|^{\min\{4, \gamma+2\}})}{e^{\lambda|y|}|y|^{d+\alpha}}dy\nonumber\\
&\leq C\int_{0}^{\delta} \rho^{\min\{3, \gamma+1\}-\alpha}d\rho\nonumber\\
&= C\frac{\delta^{\min\{4, \gamma+2\}-\alpha}}{\min\{4, \gamma+2\}-\alpha}.
\end{align}
Here and throughout, $C$ denotes a generic positive constant. This proves that $J_1$  is absolutely integrable near the origin.

For the term $I_2$, combining the analogous analysis in \eqref{eq:2.12} with $v\in C^{q}(\mathbb{R}^d)$, we derive
\begin{align}\label{r2.16}
|I_2|\leq C\left(1+\int_{\delta}^{\sqrt{d}R_0}\rho^{\min\{3, \gamma+1\}-\alpha}d\rho\right)\leq C,
\end{align}
where $0<\delta<R_0$.

For the term $I_3$, employing \eqref{eq:2.5}, \eqref{eq:2.9}, the triangle inequality, and the domain symmetry, one obtains the bound:
\begin{align}\label{eq:2.13}
|I_3|&\leq C\left(\left|\int_{\mathbb{R}^d\setminus Q_{R_0}}\frac{u(x)-u(x+y)}{e^{\lambda|y|}|y|^{d+\alpha}}dy\right|+\left|\int_{\mathbb{R}^d\setminus Q_{R_0}}{\frac{w(|y|)\sum_{1\leq|k|\leq 3}\frac{D^k u(x)}{k!}y^k}{e^{\lambda|y|}|y|^{d+\alpha}}}dy\right|\right)\nonumber\\
&\leq C\left(\left|\int_{\mathbb{R}^d\setminus Q_{R_0}}\frac{u(x)-u(x+y)}{e^{\lambda|y|}|y|^{d+\alpha}}dy\right|+\left|\frac{\Delta u(x)}{2}\int_{\mathbb{R}^d\setminus Q_{R_0}}{\frac{w(|y|)(e_1\cdot y )^2}{e^{\lambda|y|}|y|^{d+\alpha}}}dy\right|\right)\nonumber\\
&\leq C\left(\left|\int_{\mathbb{R}^d\setminus Q_{R_0}}\frac{u(x)-u(x+y)}{e^{\lambda|y|}|y|^{d+\alpha}}dy\right|+\left|\frac{\Delta u(x)}{2d}\int_{\mathbb{R}^d\setminus Q_{R_0}}{\frac{w(|y|)|y|^2}{e^{\lambda|y|}|y|^{d+\alpha}}}dy\right|\right),
\end{align}
where $e_1$ is the first standard basis vector in $\mathbb{R}^d$.

By a polar coordinate transformation, $\lambda\geq 0$, the compact support of $u\in C^s(\mathbb{R}^d)$, and $w(\rho)=0$ for $\rho\geq L$, we have
\begin{align}
|I_3|\leq C\left(\int_{R_0}^\infty\rho^{-\alpha-1}d\rho+\int_{\delta}^{L}\rho^{1-\alpha}d\rho\right)\leq C.
\end{align}
The above arguments establish the assertion in part (ii).

Part (iii). Using the domain symmetry and $w(\rho)=0$ for $\rho\geq L$, it follows from \eqref{eq:2.4} that
\begin{align}\label{eq:2.6}
J_2 &=-\frac{C_{d,\alpha}\Delta u(x)}{2}\int_{\mathbb{R}^d}{\frac{w(|y|)(e_1\cdot y )^2}{e^{\lambda|y|}|y|^{d+\alpha}}}dy\nonumber\\
&=-\frac{C_{d,\alpha}\Delta u(x)}{2d}\int_{[-L,L]^d}{\frac{w(|y|)}{e^{\lambda|y|}|y|^{d+\alpha-2}}}dy\nonumber\\
&:=-\frac{C_{d,\alpha}\Delta u(x)}{2d}K_1,
\end{align}
where $K_1$ is defined as the last integral.
Passing to polar coordinates, $K_1$ admits the explicit representation
\begin{align}\label{eq:2.7}
K_1 = \int_{[-L,L]^d}{\frac{w(|y|)}{e^{\lambda|y|}|y|^{d+\alpha-2}}}dy= \frac{2\pi^{\frac{d}{2}}}{\Gamma\left(\frac{d}{2}\right)}\int_{0}^{L}\frac{w(\rho)}{e^{\lambda \rho}\rho^{\alpha-1}}d\rho.
\end{align}
Equations \eqref{eq:2.6} and \eqref{eq:2.7} reduce $J_2$ to a one-dimensional radial integral. For the polynomial window in \eqref{eq:3.19}, this integral is a finite linear combination of incomplete Gamma functions and can therefore be evaluated directly. This completes the proof.
\end{proof}


\subsection{WIRFD Scheme}
The regularity assumptions $s\geq 8$ and $\gamma\geq 4$ in \cref{L2.1} imply $v\in C^q(\mathbb{R}^d)$ with $q\geq4$. For $d=1, 2$, the integrand $\frac{v(y)}{e^{\lambda|y|}|y|^{d+\alpha}}$ vanishes at $y=0$. This, together with the regularity of $v$, renders the punctured trapezoidal rule admissible for $J_1$. For $d=3$, the integrand need not admit a finite continuous extension at the origin; nevertheless, the punctured trapezoidal rule remains applicable, as justified by the cutoff argument in \cref{T3.1}. The regularity of $u$ further allows the use of second-order central differences for $\Delta$ in both $J_1$ and $J_2$. We thus discretize as follows.

For simplicity, take $\Omega=(-1,1)^d$ and discretize it uniformly with a spatial step size $h = \frac{2}{N_1}$, where $N_1$ is a positive integer. Define the lattice $h\mathbb{Z}^d=\{y_j = jh: j=(j_1,\ldots,j_d)\in\mathbb{Z}^d\}$. The interior grid is defined as \begin{align}\label{eq:a2.13}
\Omega_h=\{(-1+j_1h,\ldots,-1+j_dh):  1\leq j_l\leq N_1-1,\ l=1,\ldots,d\},
\end{align}
which consists of $N = (N_1-1)^d$ nodes.

Exploiting the punctured trapezoidal rule for $J_1$ in \eqref{eq:2.4}, second-order central differences for $\Delta$, domain symmetry, and \cref{L2.1}, the discrete WIRFD scheme for the TFL reads:
\begin{align}\label{eq:2.14}
(-\Delta_h)_{\lambda}^{\frac{\alpha}{2}}u(x):=
&C_{d,\alpha}h^d\left(\sum_{j\neq 0}\frac{1}{e^{\lambda|y_j|}|y_j|^{d+\alpha}}u(x)-\sum_{j\neq 0}\frac{u(x+y_j)}{e^{\lambda|y_j|}|y_j|^{d+\alpha}}+\frac{\Delta_hu(x)}{2d}\sum_{j\neq 0}{\frac{w(|y_j|)}{e^{\lambda|y_j|}|y_j|^{d+\alpha-2}}}\right)\nonumber\\
&-\frac{C_{d,\alpha}\Delta_h u(x)}{2d}\int_{[-L,L]^d}{\frac{w(|y|)}{e^{\lambda|y|}|y|^{d+\alpha-2}}}dy\nonumber\\
=& C_{d,\alpha}h^d \left(\left(S_{1,h}-S_1\right)\Delta_hu(x)+S_2 u(x)-\sum_{j\neq 0}\frac{u(x+y_j)}{e^{\lambda|y_j|}|y_j|^{d+\alpha}}\right), \quad \forall x\in \Omega_h,
\end{align}
where
\begin{align}\label{eq:2.15}
&S_{1,h}=\frac{1}{2d}\sum_{j\neq 0}{\frac{w(|y_j|)}{e^{\lambda|y_j|}|y_j|^{d+\alpha-2}}},\quad S_1= \frac{1}{2dh^d}\int_{[-L,L]^d}{\frac{w(|y|)}{e^{\lambda|y|}|y|^{d+\alpha-2}}}dy,\quad
S_2=\sum_{j\neq 0}\frac{1}{e^{\lambda|y_j|}|y_j|^{d+\alpha}},
\end{align}
and the discrete Laplacian $\Delta_h$ satisfies the second-order consistency relation
\begin{align}\label{eq:2.16}
\Delta u(x)=\Delta_h u(x)+O(h^2)=\frac{1}{h^2}\sum_{l=1}^d\left[u(x-e_lh)-2u(x)+u(x+e_lh)\right]+O(h^2).
\end{align}

Define the grid function $U=\{u_j\}_{j\in[1,N_1-1]^d}$ by $u_j = u(x_j)$ for $x_j \in \Omega_h$. The WIRFD scheme \eqref{eq:2.14} can be cast in the following matrix-vector form:
\begin{align}\label{eq:02.16}
(-\Delta_h)_{\lambda}^{\frac{\alpha}{2}}U=AU=C_{d,\alpha}h^d(D+T)U,
\end{align}
where $A=C_{d,\alpha}h^d(D+T)\in \mathbb{R}^{N \times N}$ has entries $A_{r(i),p(j)}=C_{d,\alpha}h^d(D_{r(i),p(j)}+T_{r(i),p(j)})$ for multi-indices $i,j\in[1,N_1-1]^d$, with $r(i)$, $p(j)$  the corresponding linear index maps. The matrices $D$ and $T$ are specified entrywise as
\begin{align}\label{eq:2.17}
D_{r(i),p(j)}=d_{i,j}=\begin{cases}\dfrac{S_{1,h}-S_1}{h^2},\quad&\|i-j\|_1=1,\\[6pt]
-\dfrac{2d(S_{1,h}-S_1)}{h^2},\quad &i=j,\\[6pt]
 0,\quad &\text{otherwise},
\end{cases}\qquad
T_{r(i),p(j)}=t_{i,j}=\begin{cases}
S_2,\quad&i=j,\\
-\dfrac{1}{e^{\lambda|i-j|h}|(i-j)h|^{d+\alpha}},\quad &i\neq j,
\end{cases}
\end{align}
where $\|i-j\|_1 = \sum_{l=1}^d |i_l - j_l|$ represents the Manhattan distance in $\mathbb{R}^d$. From \eqref{eq:2.17}, $A$ and $T$ are symmetric multilevel Toeplitz matrices (block Toeplitz with Toeplitz blocks when $d=2$), while $D$ is sparse and has the same multilevel Toeplitz structure.


\section{Error analysis of the WIRFD scheme for the multidimensional TFL}
In this section, we use the aliasing formula to derive the $l^{\infty}$-norm truncation error of the WIRFD scheme and describe its algorithmic procedure. The maximum norm of a grid function $U$ on $\Omega_h$ is defined by $\|U\|_{l^{\infty}(\Omega_h)}:= \max_{j\in[1,N_1-1]^d}|u_j|$.

\begin{lemma}\label{l3.1}(Aliasing formula)
\cite{G2014,ZZ2024}. Assume $u \in L^1(\mathbb R^d) \cap L^2(\mathbb R^d)$ and that the distributional derivatives $D^k u \in L^1(\mathbb R^d)$ for all multi-indices $k$ with $1 \le |k|\leq \varrho$, where $\varrho\in\mathbb{Z}^+$ and $\varrho>d$. Then, for any $h>0$, the semi-discrete Fourier transform $\check{u}(\xi)$ and the continuous Fourier transform $\hat{u}(\xi)$ satisfy the aliasing formula
\begin{align}\label{eq:f1}
\check{u}(\xi)=\sum_{j\in\mathbb{Z}^d}\hat{u}\left(\xi+\frac{2\pi j}{h}\right),\quad \xi\in\left[-\frac{\pi}{h},\frac{\pi}{h}\right]^d,
\end{align}
where
$\check{u}(\xi)=h^d\sum_{x\in h\mathbb{Z}^d}u(x)e^{-i\xi\cdot x}$, and $ \hat{u}(\xi)=\int_{\mathbb{R}^d}u(x)e^{-i\xi\cdot x}dx.$
\end{lemma}


\begin{theorem}\label{T3.1} (Truncation error).
Retaining the hypotheses of \cref{L2.1}, suppose further that $u\in C^s(\mathbb{R}^d)$ with $s\geq 8$, and that $w$ satisfies $0\leq w(\rho)\leq 1$ with $w\in C^{\gamma}([0,\infty))$, $\gamma\geq 4$. Then, for the grid function $U$ associated with $u$, the WIRFD scheme \eqref{eq:2.14} for the regularized TFL \eqref{eq:2.4} admits the following truncation error bound
\begin{align}\label{eq:3.1}
\left\|(-\Delta)_{\lambda}^{\frac{\alpha}{2}}U - (-\Delta_h)_{\lambda}^{\frac{\alpha}{2}}U\right\|_{l^{\infty}(\Omega_h)}\leq C h^{4-\alpha},
\end{align}
where $C>0$ is independent of $h$, for all sufficiently small $h>0$.
\end{theorem}
\begin{proof}
Subtracting \eqref{eq:2.14} from \eqref{eq:2.4} and invoking \eqref{eq:2.5} and \eqref{m2.7}, we derive the estimate
\begin{align}\label{eq:3.2}
&(-\Delta)_{\lambda}^{\frac{\alpha}{2}}u(x)- (-\Delta_h)_{\lambda}^{\frac{\alpha}{2}}u(x)\nonumber\\
&=C_{d,\alpha}h^d\left(S_{1,h}-S_1\right)\left(\Delta u(x)-\Delta_hu(x)\right)\nonumber\\
&\quad+C_{d,\alpha}\left(\int_{\mathbb{R}^d}\frac{v(y)}{e^{\lambda|y|}|y|^{d+\alpha}}dy-h^d\left(S_2u(x)-\sum_{j\neq 0}\frac{u(x+y_j)}{e^{\lambda|y_j|}|y_j|^{d+\alpha}}+S_{1,h}\Delta u(x)\right)\right)\nonumber\\
&:= J_3+J_4,
\end{align}
where
\begin{align}\label{r3.3}
J_3 &= C_{d,\alpha}h^d\left(S_{1,h}-S_1\right)\left(\Delta u(x)-\Delta_hu(x)\right),\\ \label{eq:3.4}
J_4&= C_{d,\alpha}\left(\int_{\mathbb{R}^d}\frac{v(y)}{e^{\lambda|y|}|y|^{d+\alpha}}dy-h^d\left(S_2u(x)-\sum_{j\neq 0}\frac{u(x+y_j)}{e^{\lambda|y_j|}|y_j|^{d+\alpha}}+S_{1,h}\Delta u(x)\right)\right).
\end{align}
It remains to estimate the error terms $J_3$ and $J_4$.
In what follows, we set $L=N_\rho h$ with a finite $N_\rho\in\mathbb Z^+$ as in \cref{L2.1}.

To estimate $J_3$ in \eqref{r3.3}, we first establish bounds for $h^dS_{1}$ and $h^dS_{1,h}$.

By \eqref{eq:2.7}, \eqref{eq:2.15}, and $0\leq w(\rho)\leq 1$, $h^dS_1$ can be estimated as
\begin{align}\label{eq:3.5}
h^dS_1 = \frac{\pi^{\frac{d}{2}}}{d\Gamma\left(\frac{d}{2}\right)}\int_{0}^{L}\frac{w(\rho)}{e^{\lambda \rho}\rho^{\alpha-1}}d\rho\leq\frac{\pi^{\frac{d}{2}}L^{2-\alpha}}{(2-\alpha)d\Gamma\left(\frac{d}{2}\right)}=
\frac{\pi^{\frac{d}{2}}{N_{\rho}}^{2-\alpha}}{(2-\alpha)d\Gamma\left(\frac{d}{2}\right)}h^{2-\alpha}.
\end{align}
From \eqref{eq:2.15}, the compact support of $w$, and $0\leq w(\rho)\leq 1$, we deduce
\begin{align}\label{eq:3.6}
h^d S_{1,h}=\frac{h^d}{2d}\sum_{j\neq 0}{\frac{w(|y_j|)}{e^{\lambda|y_j|}|y_j|^{d+\alpha-2}}}=\frac{h^d}{2d}\sum_{0<|j|\leq N_{\rho}}{\frac{w(|y_j|)}{e^{\lambda|y_j|}|y_j|^{d+\alpha-2}}}\leq \frac{h^{2-\alpha}}{2d}\sum_{0<|j|\leq N_\rho}\frac{1}{|j|^{d+\alpha-2}}.
\end{align}
Define the pairwise disjoint hierarchical index sets $F_m=\left\{j\in\mathbb Z^d:\|j\|_{\infty}=m\right\}, 1\leq m\leq N_\rho$,
where $\|j\|_\infty:=\max_{1\le l\le d}|j_l|$, and let $N_{F_m}:=|F_m|$. Then
\begin{align}\label{eq:3.7}
\{j\in\mathbb Z^d:0<|j|\le N_\rho\}\subset\bigcup_{m=1}^{N_\rho}F_m.
\end{align}
The cardinality $N_{F_m}$ is given explicitly by $$N_{F_m}=|F_m|=(2m+1)^d-(2m-1)^d=\begin{cases}\ 2,\quad & d=1,\\
\ 8m,\quad& d=2,\\
\ 24m^2+2,\quad &d =3.
\end{cases}$$
Consequently, $N_{F_m}\leq C_d m^{d-1}$, where $C_1 = 2$, $C_2 = 8$, and $C_3 = 26$.
Applying \eqref{eq:3.6}, \eqref{eq:3.7} and $N_{F_m} \leq C_d m^{d-1}$, we infer
\begin{align}\label{eq:3.8}
h^d S_{1,h}\leq\frac{h^{2-\alpha}}{2d}\sum_{m=1}^{N_\rho}\sum_{j\in F_m}\frac{1}{|j|^{d+\alpha-2}}\leq\frac{h^{2-\alpha}}{2d}\sum_{m=1}^ {N_\rho}\frac{N_{F_m}}{m^{d+\alpha-2}}\leq \frac{C_d h^{2-\alpha}}{2d}\sum_{m=1}^{N_\rho}{m^{1-\alpha}}\leq \frac{C_d}{2d}C_{\alpha}N_{\rho}^{2-\alpha}h^{2-\alpha},
\end{align}
with $C_{\alpha}$ depending only on $\alpha\in(0,2)$.

Next, invoking \eqref{eq:2.16} and the bounds in \eqref{eq:3.5} and \eqref{eq:3.8}, we conclude that
\begin{align}\label{eq:3.9}
|J_3|\leq C h^d\left(S_{1,h}+S_1\right)h^2\leq Ch^{4-\alpha}.
\end{align}

We now estimate $J_4$ in \eqref{eq:3.4}. Let $\varphi(y)=\frac{v(y)}{e^{\lambda |y|}|y|^{d+\alpha}}.$
It follows from domain symmetry that
\begin{align}\label{w3.13}
J_4=C_{d,\alpha}\left(\int_{\mathbb{R}^d}\varphi(y) dy-h^d\sum_{j\neq 0}\varphi(y_j)\right).
\end{align}
Although $\varphi$ is locally integrable for $d=1,2,3$, its fourth-order derivatives need not be integrable near the origin. To apply the aliasing formula uniformly in all three dimensions, we introduce a cutoff $\eta_h\in C^{\infty}(\mathbb{R}^d)$,
\begin{align}\label{w3.14}
\eta_h(y)=
\begin{cases}
 0, & |y|\leq\frac{h}{2},\\
 \vartheta\left(\frac{2|y|}{h}-1\right), &\frac{h}{2}<|y|<h,\\
 1, &|y|\geq h,
\end{cases}
\end{align}
where
\begin{align}\label{w3.15}
\vartheta(\upsilon)=\frac{\rho_0(\upsilon)}{\rho_0(\upsilon)+\rho_0(1-\upsilon)},\quad \rho_0(\upsilon)=
\begin{cases}
0, &\upsilon\leq 0,\\
e^{-\frac{1}{\upsilon}},  &\upsilon>0.
\end{cases}
\end{align}
Setting $\phi_h(y)=\eta_h(y)\varphi(y)$ yields $\phi_h(0)=0$, as required for applying the aliasing formula to $\phi_h(y)$.

By adding $\phi_h(y)$ and subtracting it from $\varphi(y)$, we obtain
\begin{align}\label{w3.16}
\varphi(y)=\phi_h(y)+(1-\eta_h(y))\varphi(y).
\end{align}
Inserting \eqref{w3.16} into \eqref{w3.13}, and using $\phi_h(0)=0$ along with \eqref{w3.14}, we arrive at
\begin{align}\label{w3.17}
J_4&=C_{d,\alpha}\left(\int_{\mathbb{R}^d}\left(\phi_h(y)+(1-\eta_h(y))\varphi(y)\right)dy-h^d\sum_{j\neq 0}\left(\phi_h(y_j)+(1-\eta_h(y_j))\varphi(y_j)\right)\right) \nonumber\\
&=C_{d,\alpha}\left(\int_{\mathbb{R}^d}\phi_h(y)dy-h^d\sum_{j\in\mathbb{Z}^d}\phi_h(y_j)+\int_{|y|< h}(1-\eta_h(y))\varphi(y)dy\right) \nonumber\\
&:=I_4+I_5,
\end{align}
with
\begin{equation}\label{w03.18}
\begin{aligned}
I_4&=C_{d,\alpha}\left(\int_{\mathbb{R}^d}\phi_h(y)dy
-h^d\sum_{j\in\mathbb{Z}^d}\phi_h(y_j)\right),\\
I_5&=C_{d,\alpha}\int_{|y|< h}(1-\eta_h(y))\varphi(y)dy.
\end{aligned}
\end{equation}
The decomposition $J_4=I_4+I_5$ reduces the estimation of $J_4$ to that of $I_4$ and $I_5$.

For $I_4$ in \eqref{w03.18}, we require a uniform bound for the $L^1$-norm  of $h^{\alpha}D^k\phi_h(y)$ for $|k|\leq \varrho=4$, which enables us to derive the desired error estimate via the aliasing formula.
To this end, using the partition $\mathbb R^d=\left\{|y|\le \frac {h } {2}\right\}\cup\left\{\frac {h}{2}<|y|<h\right\}\cup \left\{|y|\ge h\right\},$
the $L^1$-norm  of $h^{\alpha}D^k\phi_h(y)$ decomposes as
\begin{align}\label{m3.19}
\|h^{\alpha}D^k\phi_h\|_{L^1(\mathbb{R}^d)}&=\int_{|y|\leq\frac{h}{2}}h^{\alpha}|D^k\phi_h(y)|dy+\int_{\frac{h}{2}<|y|< h}h^{\alpha}|D^k\phi_h(y)|dy+\int_{|y|\geq h}h^{\alpha}|D^k\phi_h(y)|dy.
\end{align}

In the first region $|y|\le \frac{h}{2}$, \eqref{w3.14} gives $\eta_h(y)=0$, hence $h^{\alpha}D^k\phi_h(y)=0$.

In the second region $\frac {h}{2}<|y|<h$, the Leibniz rule yields
\begin{align}\label{w3.19}
D^k\phi_h(y)=\sum_{r\le k}\binom{k}{r} D^r\eta_h(y)\, D^{k-r}\varphi(y),
\end{align}
where $0\le r_l\le k_l$ for $l=1,\ldots,d$.
The chain rule gives
\begin{align}\label{w3.20}
|D^r\eta_h(y)|\le C h^{-|r|}.
\end{align}
Employing \eqref{eq:2.5} and the property (i) of \cref{L2.1} in the definition of $\varphi(y)$ provides the estimate
\begin{align}\label{w3.21}
|D^{k-r}\varphi(y)|\le C |y|^{\tau-d-\alpha-|k-r|}\le C h^{\tau-|k|+|r|-d-\alpha}.
\end{align}
Applying \eqref{w3.20} and \eqref{w3.21} to \eqref{w3.19} and subsequently taking the $L^1$-norm of the resulting expression, we deduce that
\begin{align}\label{m3.22}
\int_{\frac{h}{2}<|y|<h}h^{\alpha}|D^k\phi_h(y)|dy\leq C h^{\tau-|k|}.
\end{align}

In the third region $|y|\geq h$ with $0<h<\delta$, for $|k|\leq 4$, combining property (i) of \cref{L2.1}, the compact supports and regularities of $u$ and $w$, \eqref{w3.14}, and \eqref{w3.16}, we have the following estimate:
\begin{align}
\int_{|y|\geq h}h^{\alpha}|D^k\phi_h(y)|dy&=\int_{h\leq|y|< \delta}h^{\alpha}|D^k\varphi(y)|dy+\int_{\delta\leq|y|\leq R_0}h^{\alpha}|D^k\varphi(y)|dy+\int_{|y|> R_0}h^{\alpha}|D^k\varphi(y)|dy\leq C.
\end{align}

Gathering the above region-wise estimates on  $\mathbb{R}^d$ in \eqref{m3.19} delivers the following estimate for $|k|\leq4$:
\begin{align}\label{w3.24}
\|h^{\alpha}D^k\phi_h\|_{L^1(\mathbb{R}^d)}\leq C.
\end{align}
Consequently, $h^{\alpha}D^k\phi_h\in L^1(\mathbb R^d)$.

In view of $\phi_h(0)=0$, \eqref{w03.18} and \eqref{w3.24}, an application of \cref{l3.1} to $\phi_h(y)$ gives
\begin{align}\label{w3.25}
I_4 = -C_{d,\alpha}\sum_{j\neq 0}\widehat{\phi_h}\left(\frac{2\pi j}{h}\right).
\end{align}
Accordingly, the estimate for $I_4$ relies on the decay rate of $\widehat{\phi_h}(\xi)$.

Choose $i$ such that $|\xi_i|=\max_{1\le l\le d}|\xi_l|$; then $\xi_i\neq0$ and $|\xi_i|\ge |\xi|/\sqrt d$. Taking $k = 4e_i$, the Fourier transform satisfies $\widehat{\partial_{y_i}^4\phi_h}(\xi) = (i\xi_i)^4 \widehat{\phi_h}(\xi)$. Combining this with \eqref{w3.24} produces the estimate
\begin{align}\label{w3.26}
|\widehat{\phi_h}(\xi)|= |\xi_i|^{-4}\left|\int_{\mathbb{R}^d}\partial_{y_i}^4\phi_h(y)e^{-i\xi\cdot y}dy\right|\leq|\xi_i|^{-4}h^{-\alpha} \|h^{\alpha}\partial_{y_i}^4\phi_h\|_{L^1(\mathbb{R}^d)} \le C |\xi|^{-4} h^{-\alpha}.
\end{align}
Using \eqref{w3.26} in \eqref{w3.25} yields the key estimate
\begin{align}\label{w3.27}
|I_4|\leq C h^{4-\alpha}\sum_{j\neq0}\frac{1}{|2\pi j|^4}\leq C h^{4-\alpha}.
\end{align}

For $I_5$ in \eqref{w03.18}, \eqref{w3.14} and \eqref{w3.15} imply $0\leq\eta_h(y)\leq 1$. By virtue of this fact and property (i) of \cref{L2.1}, we have
\begin{align}\label{w3.28}
|I_5|\leq \int_{|y|< h}|\varphi(y)|dy\leq C\int_{0}^h\rho^{\tau-1-\alpha}d\rho\leq C h^{4-\alpha}.
\end{align}
Combining \eqref{eq:3.2}, \eqref{eq:3.9}, \eqref{w3.17}, \eqref{w3.27}, and \eqref{w3.28} completes the proof of \eqref{eq:3.1}.
\end{proof}
The validity of the preceding truncation error analysis hinges upon a compactly supported window function $w \in C^{\gamma}([0,\infty))$ obeying $$w(\rho) = 1 + O(\rho^{\gamma}), \quad \rho\to 0, \quad 0\leq w(\rho) \leq1.$$
The explicit construction of such a function is provided in \cref{r3.1}, whereas the verification of the conditions required by the above theorem is deferred to \cref{r4.1} in Section 4.

\begin{remark}\label{r3.1}(Window function). For the convergence rate in \cref{T3.1} and simplicity of the window function construction, it suffices to set $\gamma = 4$. Imposing the nine conditions $w(0)=1,\ w^{(\ell)}(0)=0\ (\ell=1,2,3),\ w^{(\ell)}(L)=0\ (\ell=0,1,2,3,4)$ uniquely determines the explicit polynomial window function
\begin{align}\label{eq:3.19}
w(\rho)=\begin{cases}
1-70\left(\frac{\rho}{L}\right)^4+224\left(\frac{\rho}{L}\right)^5-280\left(\frac{\rho}{L}\right)^6+160\left(\frac{\rho}{L}\right)^7-35\left(\frac{\rho}{L}\right)^8,\quad &0\leq\rho<L,\\
0,\quad&\rho\geq L,
\end{cases}
\end{align}
where $L:=N_\rho h=20h$. While this construction is not unique, choosing $\gamma=4$ suffices to achieve the desired convergence rate. Any $\gamma>4$ would only increase the polynomial degree without improving the convergence rate itself, since that rate is partially governed by the constant term of $w$ and the second-order central differences in \cref{T3.1}.
\end{remark}
Using the window function from \cref{r3.1}, \cref{alg:1} implements the WIRFD scheme for the TFL.
\begin{algorithm}[H]
  \caption{WIRFD scheme for the multidimensional TFL}
  \label{alg:1}
\begin{itemize}
\item{\bf{Input: }}$d\in\{1,2,3\}$, $\lambda\geq 0$, $\alpha\in(0,2)$, $h>0$, $U\in \mathbb{R}^{N}$ on $\Omega_h$.\quad {\bf{Output: }}$(-\Delta_h)_{\lambda}^{\frac{\alpha}{2}}U\in\mathbb{R}^N $ on $\Omega_h$.
\end{itemize}
\begin{algorithmic}
\STATE {\textbf{Step 1: }}Construct the grid $\Omega_h$ as defined in \eqref{eq:a2.13}.
\STATE {\textbf{Step 2: }}Compute $S_1$ from \eqref{eq:2.7} and \eqref{eq:2.15}, $\displaystyle S_1=\frac{\pi^{\frac{d}{2}}}{dh^d\Gamma\left(\frac{d}{2}\right)}\int_{0}^{L}\frac{w(\rho)}{e^{\lambda \rho}\rho^{\alpha-1}}d\rho$, where the integral is evaluated using the gamma and lower incomplete gamma functions (gammainc) in MATLAB.
\STATE {\textbf{Step 3: }}Evaluate $S_{1,h}$ based on \eqref{eq:2.15} and \eqref{eq:3.19} with $N_{\rho}=20$,
$\displaystyle S_{1,h}=\frac{1}{2d}\sum_{0<|j|\leq N_{\rho}}{\frac{w(|jh|)}{e^{\lambda|jh|}|jh|^{d+\alpha-2}}}$.
\STATE {\textbf{Step 4: }}For $d=1$, we compute $S_2$ using MATLAB's function \texttt{polylog}. For $d=2$ and $d=3$, we approximate $S_2$ in \eqref{eq:2.15} by
$\displaystyle S_2 \approx \sum_{\substack{ |j_l| \leq K\\j \neq 0}} \frac{1}{e^{\lambda |jh|} |jh|^{d+\alpha}},$
where $j=(j_1,\ldots,j_d)$, $|j_l| \leq K$ for $l=1,\ldots,d$, and $K = \left\lfloor \frac{R}{h}\right\rfloor$. The truncation radius $R$ is chosen so that the neglected tail is below the prescribed tolerance; $R=20$ is used in the reported experiments. The sum is evaluated via a parallel \texttt{parfor} loop in MATLAB.
\STATE {\textbf{Step 5: }}Evaluate the first rows of $D$ and $T$ via \eqref{eq:2.17}, then assemble the first row of $A=C_{d,\alpha}h^d(D+T)$.
\STATE {\textbf{Step 6: }}Implement $(-\Delta_h)^{\alpha/2}_\lambda U=A U$ by the fast Fourier transform (FFT), following Ref.~\cite{CDH2025}.
\end{algorithmic}
\end{algorithm}



\section{Error analysis of the WIRFD approximation for the \texorpdfstring{$d=1$}{d=1} TFL equation}
This section applies the WIRFD scheme \eqref{eq:2.14} to the TFL equation \eqref{eq:1} and derives $l^2$- and $l^\infty$-norm error estimates for the numerical approximation in the case $d=1$. Specifically, we seek $\{u_j^h\}$ satisfying
\begin{alignat}{2}\label{eq:4.1}
(-\Delta_h)_{\lambda}^{\frac{\alpha}{2}}u_j^h&=f(x_j),\quad &&x_j\in\Omega_h,\quad \alpha\in(0,2),\nonumber\\
u_j^h&=0,\quad &&x_j\in h\mathbb{Z}\setminus\Omega_h,
\end{alignat}
with $u_j^h=u^h(x_j)$ for $x_j\in h\mathbb{Z}$.
Substituting \eqref{eq:02.16} into \eqref{eq:4.1} yields the discrete linear system $AU_h=f$, where $U_h=\{u_j^h\}_{x_j\in \Omega_h}\in\mathbb{R}^N$ and $f=\{f_j\}_{x_j\in \Omega_h}\in \mathbb{R}^N$ denote the numerical solution and the source grid function, respectively.

The following lemma establishes the strict diagonal dominance of the discrete TFL matrix, a property essential for subsequent stability and error analysis.

\begin{lemma}(Strict diagonal dominance).\label{l4.1}
Under the assumptions that $d=1$, $\lambda\geq0$, $\alpha\in[1,2)$, $0\leq w(\rho)\leq 1$, and $w'(\rho)<0$ for $0<\rho< L=N_{\rho}h$, the matrix $A$ from \eqref{eq:02.16} has the following properties:
\begin{align}
a_{i,i}>0,\quad a_{i,j}<0, \quad i\neq j, \quad a_{i,i}+\sum_{\substack{j=1 \\ j \neq i}}^{N}a_{i,j}>c_0>0,
\end{align}
where $c_0:=\frac{2C_{1,\alpha}(1-2^{-\alpha})}{\alpha e^{2\lambda L_{\Omega}}L_{\Omega}^{\alpha}}$ and $L_{\Omega}:=(N+1)h$. Hence $A$ is symmetric positive definite and its minimum eigenvalue satisfies $\lambda_{\min}(A)>c_0$.
\end{lemma}
\begin{proof}
(i) Positivity of the diagonal entries.  Combining \eqref{eq:02.16} and \eqref{eq:2.17} yields
\begin{align}\label{eq:4.4}
a_{i,i}=C_{d,\alpha} h^d\left(S_2-\dfrac{2d(S_{1,h}-S_1)}{h^2}\right).
\end{align}
Recalling \eqref{eq:2.7} and \eqref{eq:2.15} with $d=1$, we deduce
\begin{align}\label{eq:4.5}
S_{1,h}-S_1&=\frac{1}{2}\sum_{j\neq 0}{\frac{w(|y_j|)}{e^{\lambda|y_j|}|y_j|^{\alpha-1}}}-\frac{1}{2h}\int_{[-L,L]}{\frac{w(|y|)}{e^{\lambda|y|}|y|^{\alpha-1}}}dy\nonumber\\
& = \sum_{j=1}^{N_{\rho}}{\frac{w(|y_j|)}{e^{\lambda|y_j|}|y_j|^{\alpha-1}}}-\frac{1}{h}\int_{0}^{L}\frac{w(\rho)}{e^{\lambda \rho}\rho^{\alpha-1}}d\rho\nonumber\\
&=\sum_{j=1}^{N_\rho}{\frac{w(|y_j|)}{e^{\lambda|y_j|}|y_j|^{\alpha-1}}}-\frac{1}{h}\sum_{j=1}^{N_\rho}\int_{y_{j-1}}^{y_{j}}\frac{w(\rho)}{e^{\lambda \rho}\rho^{\alpha-1}}d\rho\nonumber\\
&=\sum_{j=1}^{N_\rho}F_{\alpha,\lambda}(y_j)-\frac{1}{h}\sum_{j=1}^{N_\rho}\int_{y_{j-1}}^{y_{j}}F_{\alpha,\lambda}(\rho)d\rho,
\end{align}
where $F_{\alpha,\lambda}(\rho):=\frac{w(\rho)}{e^{\lambda \rho}\rho^{\alpha-1}}$.
Differentiating $F_{\alpha,\lambda}(\rho)$ with respect to $\rho$ gives
\begin{align}\label{eq:4.6}
F_{\alpha,\lambda}'(\rho)&:=w(\rho)e^{-\lambda \rho}(1-\alpha)\rho^{-\alpha}+(w'(\rho)e^{-\lambda \rho}-\lambda w(\rho)e^{-\lambda \rho})\rho^{1-\alpha}\nonumber\\
&=((1-\alpha-\lambda \rho)w(\rho)+w'(\rho)\rho)e^{-\lambda \rho}\rho^{-\alpha}.
\end{align}

Under $0\leq w(\rho)\leq 1$ and $w'(\rho)<0$, \eqref{eq:4.6} gives $F_{\alpha,\lambda}'(\rho)<0$ on $(0,L)$ for each $\alpha\in[1,2)$. The strict decrease of $F_{\alpha,\lambda}(\rho)$ on $(0,L)$ implies, for each interval $[y_{j-1},y_j]$ with $j=1,\ldots,N_\rho$,
\begin{align}\label{eq:4.7}
F_{\alpha,\lambda}(y_j)<\frac{1}{h}\int_{y_{j-1}}^{y_j}F_{\alpha,\lambda}(\rho)d\rho.
\end{align}
Substituting \eqref{eq:4.7} into \eqref{eq:4.5} yields $S_{1,h}-S_1<0$. Coupled with $S_2>0$ from \eqref{eq:2.15} and \eqref{eq:4.4}, these two estimates ensure $a_{i,i}>0$.

(ii) Negativity of off-diagonal entries. Using \eqref{eq:02.16}, \eqref{eq:2.17}, and $S_{1,h}-S_1<0$, we infer  $a_{i,j}=C_{1,\alpha}h(d_{i,j}+t_{i,j})<0$ for $i\neq j$.

(iii) Strictly diagonally dominant. Upon summing the diagonal and off-diagonal entries of \eqref{eq:02.16} while invoking \eqref{eq:2.15} and \eqref{eq:2.17}, we obtain the following inequality
\begin{align}\label{eq:4.8}
a_{i,i}+\sum_{\substack{j=1 \\ j \neq i}}^{N}a_{i,j}&\geq C_{1,\alpha}h\left( S_2-2\sum_{j=1}^{N-1}\frac{1}{e^{\lambda|y_j|}|y_j|^{1+\alpha}}\right)\geq2C_{1,\alpha}h\sum_{j=N}^{\infty}\frac{1}{e^{\lambda y_{j}}y_{j}^{1+\alpha}}=2C_{1,\alpha}h\sum_{j=N}^{\infty}F(y_j),
\end{align}
where $F(y)=\dfrac{1}{e^{\lambda y}y^{1+\alpha}}$.
Since $F(y)$ is strictly decreasing on $[y_N,\infty)$, for each subinterval  $[y_j,y_{j+1}]$ with $j\geq N$, the following inequality holds:
\begin{align}\label{eq:4.9}
h F(y_j) > \int_{y_j}^{y_{j+1}}F(y)dy.
\end{align}
Plugging \eqref{eq:4.9} into \eqref{eq:4.8} leads to
\begin{align}
a_{i,i}+\sum_{\substack{j=1 \\ j \neq i}}^{N}a_{i,j}>2C_{1,\alpha}\int_{y_N}^{\infty}F(y)dy> 2C_{1,\alpha}\int_{L_{\Omega}}^{2L_{\Omega}}F(y)dy\geq \frac{2C_{1,\alpha}(1-2^{-\alpha})}{\alpha e^{2\lambda L_{\Omega}}L_{\Omega}^{\alpha}},
\end{align}
where $L_{\Omega}=(N+1)h$.

By Gershgorin's theorem (\cite{A1996}, Theorem 4.4) and the matrix properties (i)-(iii), the minimum eigenvalue $\lambda_{\min}(A)$ has the following lower bound:
\begin{align}
\lambda_{\min}(A)\geq\min_{1\leq i\leq N}\left(a_{i,i}-\sum_{\substack{j=1 \\ j \neq i}}^{N}|a_{i,j}|\right)=\min_{1\leq i\leq N}\left(a_{i,i}+\sum_{\substack{j=1 \\ j \neq i}}^{N}a_{i,j}\right)>\frac{2C_{1,\alpha}(1-2^{-\alpha})}{\alpha e^{2\lambda L_{\Omega}}L_{\Omega}^{\alpha}}>0.
\end{align} Therefore, $A$ is symmetric positive definite, and the WIRFD approximation \eqref{eq:4.1} possesses a unique solution.
\end{proof}
\begin{remark}\label{r4.1}
The window function in \eqref{eq:3.19} has $w'(\rho)=-\frac{280\rho^3}{L^4}\left(1-\frac{\rho}{L}\right)^4<0$ for $\rho\in(0,L)$ with $w(0)=1$ and $w(L)=0$. Hence the preceding results apply.
For $d=2,3$, extending \cref{l4.1} requires a proof of $S_{1,h}-S_1<0$ for a Cartesian lattice sum relative to a radial integral. The lack of a natural radial ordering of lattice points makes the one-dimensional monotonicity argument unavailable, particularly for $\alpha\in(0,1)$. We therefore restrict the rigorous stability and convergence analysis to $d=1$. Figure~\ref{fig:0} nevertheless provides numerical evidence that $S_{1,h}-S_1<0$ in $d=2$ for $\lambda=1.0$ and $\alpha=0.4,0.7,1.0,1.3,1.6,1.9$.
\end{remark}
\begin{figure}[H]
\centering
\includegraphics[width=0.49\linewidth]{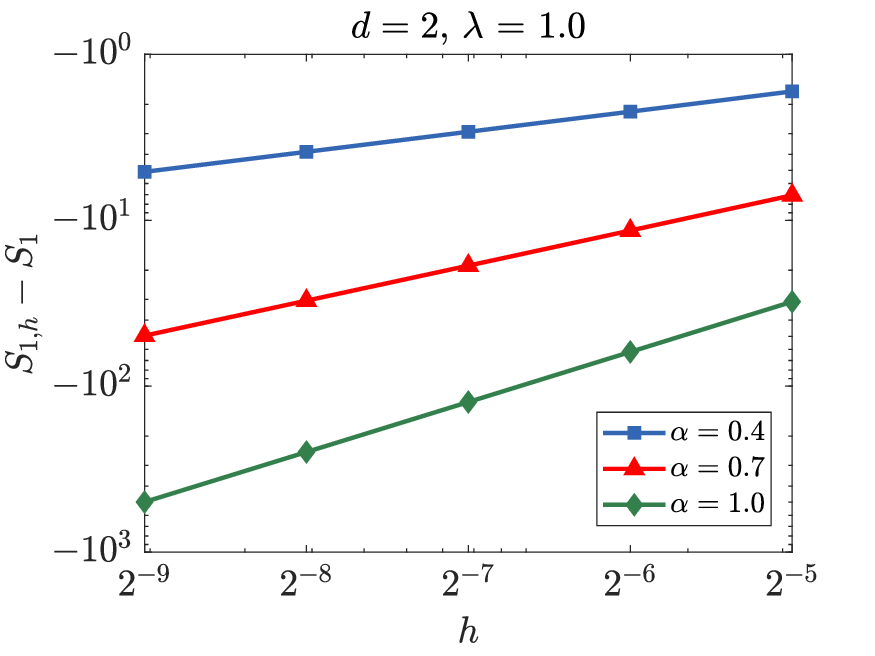}
\includegraphics[width=0.49\linewidth]{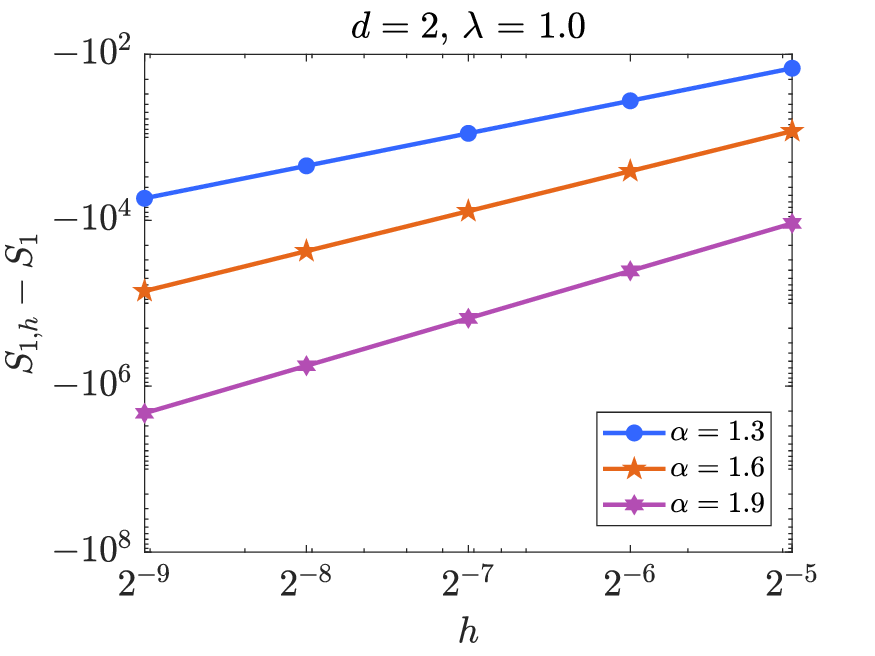}
\caption{Strict inequality $S_{1,h}-S_1<0$ ($\lambda =1.0$, $\alpha = 0.4$, $0.7$, $1.0$, $1.3$, $1.6$, $1.9$).}
\label{fig:0}
\end{figure}

We now analyze the stability of the numerical method \eqref{eq:4.1} for $d=1$ and $\alpha\in[1,2)$. The $l^2$ inner product and its associated norm are defined as
\begin{align}
(U,V)_{\Omega_h}=h\sum_{x_j\in\Omega_h}u_j\bar{v}_j, \qquad \|U\|_{l^2(\Omega_h)}=\sqrt{(U,U)_{\Omega_h}},
\end{align}
where $U=\{u_j\}_{x_j\in\Omega_h}$ and $V=\{v_j\}_{x_j\in\Omega_h}$ are grid functions, and $\bar{v}_j$ is the complex conjugate of $v_j$.

\begin{theorem}\label{d4.1} (Stability). Under the assumptions of \cref{l4.1} and $f\in l^{\infty}(\Omega_h)$, the numerical solution $U_h$ of \eqref{eq:4.1} satisfies the bounds
\begin{align}\label{eq:4.11}
\|U_h\|_{l^{2}(\Omega_h)}\leq \frac{1}{c_0}\|f\|_{l^2(\Omega_h)}, \qquad \|U_h\|_{l^{\infty}(\Omega_h)}\leq \frac{1}{c_0} \|f\|_{l^\infty(\Omega_h)},
\end{align}
with $c_0>0$ from \cref{l4.1}.
\end{theorem}
\begin{proof}
We first establish a uniform bound on the $l^2$-norm of $U_h$.

Testing \eqref{eq:4.1} against $U_h$ in the $l^2$ inner product, and applying the symmetric positive definiteness from \cref{l4.1} together with the Cauchy-Schwarz inequality, we deduce
\begin{align}\label{4.12}
\lambda_{\min}(A)\|U_h\|_{l^2(\Omega_h)}^2\leq(AU_h,U_h)_{\Omega_h}\leq \|f\|_{l^2(\Omega_h)}\|U_h\|_{l^2(\Omega_h)}.
\end{align}
Incorporating the minimum eigenvalue lower bound from \cref{l4.1} into \eqref{4.12} leads to the following estimate:
\begin{align}
\|U_h\|_{l^2(\Omega_h)}\leq \frac{1}{\lambda_{\min}(A)}\|f\|_{l^2(\Omega_h)}\leq\frac{1}{c_0}\|f\|_{l^2(\Omega_h)}.
\end{align}

Next, we derive a uniform bound on the $l^\infty$-norm of $U_h$.

Let $\|U_h\|_{l^\infty(\Omega_h)}=\left|u_{i_0}^h\right|$ for some $1\leq i_0\leq N$. At this index, we have the equality
\begin{align}\label{eq:4.14}
\left(f_{i_0}-c_0u_{i_0}^h\right)u_{i_0}^h&=\left(\sum_{j=1}^{N}a_{i_0,j}u_{j}^h-c_0u_{i_0}^h\right)u_{i_0}^h.
\end{align}
Since $\|U_h\|_{l^\infty(\Omega)}=\left|u_{i_0}^h\right|$,  we obtain $u_j^hu_{i_0}^h\leq\left|u_j^hu_{i_0}^h\right|\leq \left(u_{i_0}^h\right)^2$. Together with $a_{i_0,j} < 0$ for $j\neq i_0$ by \cref{l4.1}, this implies
\begin{align}\label{eq:04.15}
a_{i_0,j}u_j^hu_{i_0}^h\geq a_{i_0,j}\left(u_{i_0}^h\right)^2.
\end{align}

With the aid of \cref{l4.1} and \eqref{eq:04.15}, we deduce from \eqref{eq:4.14} that
\begin{align}\label{eq:4.15}
\left(f_{i_0}-c_0u_{i_0}^h\right)u_{i_0}^h\geq\left(\sum_{j=1}^{N}a_{i_0,j}-c_0\right)\left(u_{i_0}^h\right)^2>0.
\end{align}
The positivity of $c_0$ and \eqref{eq:4.15} imply $\left(f_{i_0}-c_0u_{i_0}^h\right)c_0u_{i_0}^h>0$, indicating the two factors share the same sign. Therefore, $c_0\left|u_{i_0}^h\right|\leq\left|f_{i_0}-c_0u_{i_0}^h+c_0u_{i_0}^h\right|=\left|f_{i_0}\right|$, which proves \eqref{eq:4.11}.
\end{proof}

The error estimates build upon the preceding stability analysis.
\begin{theorem}\label{D4.2} (Error estimates).
Under the assumptions of \cref{T3.1} and \cref{d4.1}, the exact grid solution $U$ of \eqref{eq:1} and the numerical solution $U_h$ of \eqref{eq:4.1} satisfy the following error bounds for $d=1$, $\alpha\in[1,2)$, and sufficiently small $h>0$:
\begin{align}\label{eq:4.16}
\|U-U_h\|_{l^2(\Omega_h)}\leq Ch^{4-\alpha},\qquad \|U-U_h\|_{l^\infty(\Omega_h)}\leq Ch^{4-\alpha},
\end{align}
with $C$ independent of $h$.
\end{theorem}

\begin{proof}
Subtracting \eqref{eq:4.1} from \eqref{eq:1} at each $x_j\in\Omega_h$ gives the error equation
\begin{align}\label{eq:4.17}
(-\Delta_h)_{\lambda}^{\frac{\alpha}{2}}\left(U-U_h\right)=A\left(U-U_h\right)=(-\Delta_h)_{\lambda}^{\frac{\alpha}{2}}U-(-\Delta)_{\lambda}^{\frac{\alpha}{2}}U,
\end{align}
with $u_j-u_j^h=0$ for $x_j\in h\mathbb{Z}\setminus\Omega_h$.
Taking the $l^2$ inner product of \eqref{eq:4.17} with $U-U_h$, and leveraging \cref{T3.1} and \cref{d4.1}, we recover the estimate
\begin{align}\label{eq:4.21}
\|U-U_h\|_{l^2(\Omega_h)}&\leq \frac{1}{c_0}\left\|(-\Delta_h)_{\lambda}^{\frac{\alpha}{2}}U-(-\Delta)_{\lambda}^{\frac{\alpha}{2}}U\right\|_{l^2(\Omega_h)}\leq C\left\|(-\Delta)_{\lambda}^{\frac{\alpha}{2}}U - (-\Delta_h)_{\lambda}^{\frac{\alpha}{2}}U\right\|_{l^{\infty}(\Omega_h)}\leq C h^{4-\alpha}.
\end{align}

The $l^{\infty}$-norm error estimate is given as follows.

Under the condition $\|U-U_h\|_{l^\infty(\Omega_h)}=\left|u_{i_0}-u_{i_0}^h\right|$, repeating the argument in \eqref{eq:4.14}--\eqref{eq:4.15}, with $u_j^h$ and $u_{i_0}^h$ replaced by $u_j-u_j^h$ and $u_{i_0}-u_{i_0}^h$, respectively, we have
\begin{align}\label{eq:4.22}
\|U-U_h\|_{l^\infty(\Omega_h)}\leq C\left\|(-\Delta)_{\lambda}^{\frac{\alpha}{2}}U-(-\Delta_h)_{\lambda}^{\frac{\alpha}{2}}U\right\|_{l^{\infty}(\Omega_h)}.
\end{align}

Applying \cref{T3.1} to \eqref{eq:4.21} and \eqref{eq:4.22} establishes the desired estimates  \eqref{eq:4.16}.
\end{proof}
Algorithm 2 outlines the main computational steps of the WIRFD method for the TFL equation.
\begin{algorithm}[H]
  \caption{WIRFD approximation for the TFL equation}
  \label{alg:2}
\begin{itemize}
\item{\bf{Input: }}$d\in\{1,2,3\}$, $\lambda\geq 0$, $\alpha\in(0,2)$, $h>0$, $f\in \mathbb{R}^{N}$ on $\Omega_h$.\quad {\bf{Output: }}$U_h\in\mathbb{R}^{N}$.
\end{itemize}
\begin{algorithmic}
\STATE {\textbf{Step 1: }} Generate a uniform grid with spatial step size $h$ according to \eqref{eq:a2.13}.
\STATE {\textbf{Step 2: }} Assemble the first row (or first block row) of the matrix $A$ using \cref{alg:1}.
\STATE {\textbf{Step 3: }} Take $M$ as the preconditioner from \cite[Section 4.1]{HZD2021}.
\STATE {\textbf{Step 4: }} Solve \eqref{eq:4.1} for $U_h$ with the conjugate gradient (CG) or preconditioned conjugate gradient (PCG) method.
\end{algorithmic}
\end{algorithm}



\section{Numerical results}
This section presents numerical results supporting the error estimates of the WIRFD scheme \eqref{eq:2.14} for the TFL and the WIRFD approximation \eqref{eq:4.1} for the TFL equation. All tests were carried out in MATLAB R2020b (64-bit) on a 2.3 GHz workstation with 16 GB RAM.

\subsection{Numerical verification of the discrete TFL}
This subsection numerically evaluates the errors and convergence rates of the WIRFD scheme from \cref{T3.1} for $u\in C^s(\mathbb{R}^d)$ with distinct $\lambda$. Using \cref{alg:1}, we compute the discrete TFL $(-\Delta_h)_{\lambda}^{\frac{\alpha}{2}}U$  for several $h$. 
In the absence of an exact closed-form expression for the TFL, the $l^2$- and $l^{\infty}$-norm errors are measured by the differences between numerical approximations on two successive grids $e_{l^2}(h)=\left\|(-\Delta_h)_{\lambda}^{\frac{\alpha}{2}}U-\left(-\Delta_{h/2}\right)_{\lambda}^{\frac{\alpha}{2}}U\right\|_{l^{2}(\Omega_h)}$ and $e_{l^\infty}(h)=\left\|(-\Delta_h)_{\lambda}^{\frac{\alpha}{2}}U-\left(-\Delta_{h/2}\right)_{\lambda}^{\frac{\alpha}{2}}U\right\|_{l^{\infty}(\Omega_h)}.$
In each two-grid difference, the fine-grid quantity is restricted to the coarse grid. The convergence rates are $\log_{2}\left(\frac{e_{l^2}(2h)}{e_{l^2}(h)}\right)$ and $\log_{2}\left(\frac{e_{l^\infty}(2h)}{e_{l^\infty}(h)}\right)$.
\begin{example}\label{e1}
Take the function $u$ supported on $\Omega=(-1,1)$, defined by
\begin{align}\label{eq:5.1}
u(x)=(1-x^2)_+^p,\quad x\in\mathbb{R},\quad p>0,
\end{align}
where $(\cdot)_+=\max\{\cdot,0\}$.
\end{example}
The parameter $p$ controls the regularity of the function. Indeed, the function is smooth in the interior of $\Omega=(-1,1)$, and its
limited regularity occurs only at the boundary of its compact support. More precisely, if $p \in \mathbb{N}$, then $u\in C^{p-1,1}(\mathbb{R})\subset C^{p-1}(\mathbb{R})$, whereas if
$p\notin\mathbb{N}$, then $u\in C^{\lfloor p\rfloor,p-\lfloor p\rfloor}(\mathbb{R})\subset C^{\lfloor p \rfloor}(\mathbb{R})$.

\cref{tab1,tab2,tab3,tab4} report the $l^2$- and $l^{\infty}$-norm errors and convergence rates of the WIRFD scheme for the $d=1$ TFL, with $h=2^{-4}, 2^{-5}, \ldots, 2^{-8}$. Two parameter sets are considered: $\lambda=0.5$ with $p=8+\alpha$ for $\alpha=0.7, 1.0, 1.3, 1.6$ in \cref{tab1,tab2}, and $\lambda=1$ with $p=8.1$ in \cref{tab3,tab4}. In all tested cases, the errors $e_{l^2}$ and $e_{l^\infty}$ decrease steadily and the observed rates agree with $O(h^{4-\alpha})$, in accordance with the truncation error estimate in \cref{T3.1}. The same order is retained for both tested values of $\lambda$.

\begin{table}[H]
\centering
\begin{threeparttable}
\renewcommand{\arraystretch}{1.15} 
   \setlength{\tabcolsep}{9.pt}
\caption{Error $e_{l^2}(h)$ and convergence rate for $(-\Delta_h)_{\lambda}^{\frac{\alpha}{2}}U$ ($d=1$, $\lambda=0.5$, $p=8+\alpha$).}
\label{tab1}
\begin{tabular}{cccccccccc}
\toprule
\multirow{2}{*}{$h$} &
\multicolumn{2}{c}{$\alpha=0.7$} &
\multicolumn{2}{c}{$\alpha=1.0$} &
\multicolumn{2}{c}{$\alpha=1.3$}&
\multicolumn{2}{c}{$\alpha=1.6$}\\
\cmidrule(lr){2-3} \cmidrule(lr){4-5} \cmidrule(lr){6-7} \cmidrule(lr){8-9}
 & $e_{l^2}(h)$ & rate & $e_{l^2}(h)$ & rate & $e_{l^2}(h)$ & rate& $e_{l^2}(h)$ &rate \\
\midrule
$2^{-4}$ & 1.41e-04 & * & 1.51e-03  & * & 2.99e-03  & * & 2.16e-02  & * \\
$2^{-5}$ & 1.45e-05 &   3.29   &   1.91e-04  & 2.98  & 4.66e-04 &  2.68& 4.13e-03 & 2.38\\
$2^{-6}$ & 1.48e-06    &   3.29&  2.40e-05 & 3.00   & 7.19e-05 & 2.70  &  7.85e-04 &  2.40\\
$2^{-7}$ & 1.51e-07 &  3.29   &   3.00e-06  &  3.00 & 1.11e-05  & 2.70  & 1.49e-04 & 2.40\\
$2^{-8}$ &1.55e-08  &   3.28   & 3.74e-07 & 3.00  &   1.69e-06  & 2.71 & 2.81e-05 & 2.40 \\
\bottomrule
\end{tabular}
\end{threeparttable}
\end{table}
\vspace{-0.6cm}
\begin{table}[H]
\centering
\begin{threeparttable}
\renewcommand{\arraystretch}{1.15} 
   \setlength{\tabcolsep}{9.pt}
\caption{Error $e_{l^{\infty}}(h)$ and convergence rate for $(-\Delta_h)_{\lambda}^{\frac{\alpha}{2}}U$ ($d=1$, $\lambda=0.5$, $p=8+\alpha$).}
\label{tab2}
\begin{tabular}{cccccccccc}
\toprule
\multirow{2}{*}{$h$} &
\multicolumn{2}{c}{$\alpha=0.7$} &
\multicolumn{2}{c}{$\alpha=1.0$} &
\multicolumn{2}{c}{$\alpha=1.3$}&
\multicolumn{2}{c}{$\alpha=1.6$}\\
\cmidrule(lr){2-3} \cmidrule(lr){4-5} \cmidrule(lr){6-7} \cmidrule(lr){8-9}
 & $e_{l^{\infty}}(h)$ & rate & $e_{l^{\infty}}(h)$ & rate & $e_{l^{\infty}}(h)$ & rate& $e_{l^{\infty}}(h)$ &rate \\
\midrule
$2^{-4}$ & 2.21e-04  &   *  &   2.41e-03   & *  & 4.83e-03 &  *& 3.52e-03 &*\\
$2^{-5}$ & 2.27e-05  &  3.29 &   3.04e-04   & 2.98  & 7.51e-04 & 2.68& 6.75e-03 &2.38\\
$2^{-6}$ & 2.32e-06     &   3.29&  3.82e-05 & 3.00   & 1.16e-04  & 2.70  &  1.28e-03 &  2.40\\
$2^{-7}$ & 2.36e-07 &  3.29   &   4.77e-06  &  3.00 &  1.78e-05   & 2.70  & 2.43e-04 & 2.40\\
$2^{-8}$ &2.43e-08   &   3.28   & 5.94e-07 & 3.01  &    2.72e-06  & 2.71 & 4.58e-05 & 2.41 \\
\bottomrule
\end{tabular}
\end{threeparttable}
\end{table}
\vspace{-0.6cm}
\begin{table}[H]
\centering
\begin{threeparttable}
\renewcommand{\arraystretch}{1.15} 
   \setlength{\tabcolsep}{9.pt}
\caption{Error $e_{l^{2}}(h)$ and convergence rate for $(-\Delta_h)_{\lambda}^{\frac{\alpha}{2}}U$ ($d=1$, $\lambda=1.0$, $p=8.1$).}
\label{tab3}
\begin{tabular}{cccccccccc}
\toprule
\multirow{2}{*}{$h$} &
\multicolumn{2}{c}{$\alpha=0.7$} &
\multicolumn{2}{c}{$\alpha=1.0$} &
\multicolumn{2}{c}{$\alpha=1.3$}&
\multicolumn{2}{c}{$\alpha=1.6$}\\
\cmidrule(lr){2-3} \cmidrule(lr){4-5} \cmidrule(lr){6-7} \cmidrule(lr){8-9}
 & $e_{l^2}(h)$ & rate & $e_{l^2}(h)$ & rate & $e_{l^2}(h)$ & rate& $e_{l^2}(h)$ &rate \\
\midrule
$2^{-4}$ & 1.25e-04 & * &   1.26e-03 & *  & 2.36e-03 &  *& 1.61e-02 & *\\
$2^{-5}$ & 1.28e-05 & 3.28 &   1.60e-04 & 2.98  & 3.68e-04 &  2.68& 3.09e-03 & 2.38\\
$2^{-6}$ & 1.31e-06 & 3.29 &  2.00e-05 & 2.99  & 5.68e-05  & 2.69  &  5.88e-04 &  2.39\\
$2^{-7}$ & 1.33e-07 & 3.29 &   2.51e-06  & 3.00 & 8.74e-06   & 2.70  & 1.11e-04 & 2.40\\
$2^{-8}$ & 1.37e-08 &   3.28   & 3.12e-07 & 3.01  &  1.34e-06  & 2.71 &  2.10e-05 & 2.41 \\
\bottomrule
\end{tabular}
\end{threeparttable}
\end{table}
\vspace{-0.6cm}
\begin{table}[H]
\centering
\begin{threeparttable}
\renewcommand{\arraystretch}{1.15} 
   \setlength{\tabcolsep}{9.pt}
\caption{Error $e_{l^{\infty}}(h)$ and convergence rate for $(-\Delta_h)_{\lambda}^{\frac{\alpha}{2}}U$ ($d=1$, $\lambda=1.0$, $p=8.1$).}
\label{tab4}
\begin{tabular}{cccccccccc}
\toprule
\multirow{2}{*}{$h$} &
\multicolumn{2}{c}{$\alpha=0.7$} &
\multicolumn{2}{c}{$\alpha=1.0$} &
\multicolumn{2}{c}{$\alpha=1.3$}&
\multicolumn{2}{c}{$\alpha=1.6$}\\
\cmidrule(lr){2-3} \cmidrule(lr){4-5} \cmidrule(lr){6-7} \cmidrule(lr){8-9}
 & $e_{l^{\infty}}(h)$ & rate & $e_{l^{\infty}}(h)$ & rate & $e_{l^{\infty}}(h)$ & rate& $e_{l^{\infty}}(h)$ &rate \\
\midrule
$2^{-4}$ & 1.89e-04  & * &  1.91e-03 & * & 3.58e-03 & *& 2.45e-02 & *\\
$2^{-5}$ & 1.94e-05  & 3.28 &   2.42e-04 & 2.98 & 5.58e-04 & 2.68& 4.69e-03 & 2.38\\
$2^{-6}$ & 1.99e-06 & 3.29 &  3.04e-05  & 2.99  &  8.62e-05 & 2.69  &  8.92e-04  &  2.39\\
$2^{-7}$ & 2.03e-07 & 3.29 &   3.81e-06  & 3.00 & 1.33e-05   & 2.70  & 1.69e-04   & 2.40\\
$2^{-8}$ & 2.09e-08  &  3.28  & 4.73e-07 & 3.01  &  2.02e-06    & 2.71 &   3.19e-05   & 2.41 \\
\bottomrule
\end{tabular}
\end{threeparttable}
\end{table}

To verify the dimensionality robustness of the error estimates in \cref{T3.1}, we now examine the $d=2$ case.
\begin{example}\label{e2}
Consider the compactly supported function $u$ in $\Omega=(-1,1)^2$ defined by
\begin{align}\label{eq:5.2}
u(x_1,x_2)=\left[(1-x_1^2)_+(1-x_2^2)_+\right]^p,\quad (x_1,x_2)\in\mathbb{R}^2,
\end{align}
where $p>0$.
\end{example}
We ensure the test function $u\in C^s(\mathbb{R}^2)$ with $s\geq 8$ by selecting $p=8+\alpha$ and $p=8.1$. For $\alpha=0.7, 1.0, 1.3, 1.6$, \cref{tab5,tab6,tab7,tab8} summarize the $l^2$- and $l^{\infty}$-norm errors along with the corresponding convergence rates for several $h=2^{-4}, 2^{-5}, \ldots, 2^{-8}$ under two parameter configurations: $\lambda=0.5, p=8+\alpha$ and $\lambda=1.0, p=8.1$. In every tested case, the observed rates agree with $O(h^{4-\alpha})$, supporting the prediction in \cref{T3.1}. The order is retained across the two choices of $\lambda$ and the tested smoothness parameters.

\begin{table}[H]
\centering
\begin{threeparttable}
\renewcommand{\arraystretch}{1.15} 
   \setlength{\tabcolsep}{9.pt}
\caption{Error $e_{l^2}(h)$ and convergence rate for $(-\Delta_h)_{\lambda}^{\frac{\alpha}{2}}U$ ($d=2$, $\lambda=0.5$, $p=8+\alpha$).}
\label{tab5}
\begin{tabular}{cccccccccc}
\toprule
\multirow{2}{*}{$h$} &
\multicolumn{2}{c}{$\alpha=0.7$} &
\multicolumn{2}{c}{$\alpha=1.0$} &
\multicolumn{2}{c}{$\alpha=1.3$}&
\multicolumn{2}{c}{$\alpha=1.6$}\\
\cmidrule(lr){2-3} \cmidrule(lr){4-5} \cmidrule(lr){6-7} \cmidrule(lr){8-9}
 & $e_{l^2}(h)$ & rate & $e_{l^2}(h)$ & rate & $e_{l^2}(h)$ & rate& $e_{l^2}(h)$ &rate \\
\midrule
$2^{-4}$ & 6.14e-05 & * & 9.43e-04  & * & 1.10e-03   & * & 7.49e-03  & * \\
$2^{-5}$ & 6.31e-06 &  3.28 &   1.19e-04  & 2.98  & 1.72e-04 &  2.68& 1.44e-03 & 2.38\\
$2^{-6}$ & 6.42e-07 & 3.30 &  1.50e-05 & 3.00   & 2.65e-05 & 2.69 & 2.73e-04 &  2.39\\
$2^{-7}$ & 6.48e-08 & 3.31  &   1.86e-06  &  3.01  &4.08e-06& 2.70  & 5.18e-05 & 2.40\\
$2^{-8}$ & 6.37e-09 & 3.35 & 2.28e-07 & 3.03  & 6.22e-07  & 2.71 &  9.79e-06  & 2.40 \\
\bottomrule
\end{tabular}
\end{threeparttable}
\end{table}
\vspace{-0.6cm}
\begin{table}[H]
\centering
\begin{threeparttable}
\renewcommand{\arraystretch}{1.15} 
   \setlength{\tabcolsep}{9.pt}
\caption{Error $e_{l^{\infty}}(h)$ and convergence rate for $(-\Delta_h)_{\lambda}^{\frac{\alpha}{2}}U$ ($d=2$, $\lambda=0.5$, $p=8+\alpha$).}
\label{tab6}
\begin{tabular}{cccccccccc}
\toprule
\multirow{2}{*}{$h$} &
\multicolumn{2}{c}{$\alpha=0.7$} &
\multicolumn{2}{c}{$\alpha=1.0$} &
\multicolumn{2}{c}{$\alpha=1.3$}&
\multicolumn{2}{c}{$\alpha=1.6$}\\
\cmidrule(lr){2-3} \cmidrule(lr){4-5} \cmidrule(lr){6-7} \cmidrule(lr){8-9}
 & $e_{l^{\infty}}(h)$ & rate & $e_{l^{\infty}}(h)$ & rate & $e_{l^{\infty}}(h)$ & rate& $e_{l^{\infty}}(h)$ &rate \\
\midrule
$2^{-4}$ &2.75e-04& * & 4.36e-03   & * & 5.24e-03  & * & 3.66e-02 & * \\
$2^{-5}$ & 2.83e-05 &   3.28   &   5.51e-04 & 2.98  & 8.17e-04 &  2.68& 7.02e-03 & 2.38\\
$2^{-6}$ & 2.88e-06 & 3.30 &  6.91e-05  & 3.00 & 1.26e-04 & 2.70 &  1.34e-03 &  2.39\\
$2^{-7}$ & 2.90e-07 & 3.31 &   8.60e-06 & 3.01 & 1.94e-05  & 2.70  & 2.53e-04 & 2.40\\
$2^{-8}$ & 2.83e-08 & 3.35 & 1.05e-06 & 3.04 & 2.95e-06  & 2.72 & 4.79e-05 & 2.40 \\
\bottomrule
\end{tabular}
\end{threeparttable}
\end{table}
\vspace{-0.6cm}
\begin{table}[H]
\centering
\begin{threeparttable}
\renewcommand{\arraystretch}{1.15} 
   \setlength{\tabcolsep}{9.pt}
\caption{Error $e_{l^{2}}(h)$ and convergence rate for $(-\Delta_h)_{\lambda}^{\frac{\alpha}{2}}U$ ($d=2$, $\lambda=1.0$, $p=8.1$).}
\label{tab7}
\begin{tabular}{cccccccccc}
\toprule
\multirow{2}{*}{$h$} &
\multicolumn{2}{c}{$\alpha=0.7$} &
\multicolumn{2}{c}{$\alpha=1.0$} &
\multicolumn{2}{c}{$\alpha=1.3$}&
\multicolumn{2}{c}{$\alpha=1.6$}\\
\cmidrule(lr){2-3} \cmidrule(lr){4-5} \cmidrule(lr){6-7} \cmidrule(lr){8-9}
 & $e_{l^2}(h)$ & rate & $e_{l^2}(h)$ & rate & $e_{l^2}(h)$ & rate& $e_{l^2}(h)$ &rate \\
\midrule
$2^{-4}$ & 5.48e-05 & * &  8.02e-04  & * &  8.96e-04  & * & 5.83e-03& * \\
$2^{-5}$ & 5.66e-06 & 3.28 &  1.02e-04  & 2.98  & 1.40e-04 &  2.68& 1.12e-03 & 2.38\\
$2^{-6}$ & 5.77e-07 & 3.29 &  1.28e-05 & 2.99  & 2.17e-05 & 2.69  & 2.13e-04 &  2.39\\
$2^{-7}$ &5.83e-08 & 3.31 &   1.59e-06  & 3.00 & 3.33e-06  & 2.70  & 4.04e-05  & 2.40\\
$2^{-8}$ &  5.72e-09 &   3.35  & 1.95e-07 & 3.03  &  5.07e-07   & 2.72 &  7.64e-06  & 2.40 \\
\bottomrule
\end{tabular}
\end{threeparttable}
\end{table}
\vspace{-0.6cm}
\begin{table}[H]
\centering
\begin{threeparttable}
\renewcommand{\arraystretch}{1.15} 
   \setlength{\tabcolsep}{9.pt}
\caption{Error $e_{l^{\infty}}(h)$ and convergence rate for $(-\Delta_h)_{\lambda}^{\frac{\alpha}{2}}U$ ($d=2$, $\lambda=1.0$, $p=8.1$).}
\label{tab8}
\begin{tabular}{cccccccccc}
\toprule
\multirow{2}{*}{$h$} &
\multicolumn{2}{c}{$\alpha=0.7$} &
\multicolumn{2}{c}{$\alpha=1.0$} &
\multicolumn{2}{c}{$\alpha=1.3$}&
\multicolumn{2}{c}{$\alpha=1.6$}\\
\cmidrule(lr){2-3} \cmidrule(lr){4-5} \cmidrule(lr){6-7} \cmidrule(lr){8-9}
 & $e_{l^{\infty}}(h)$ & rate & $e_{l^{\infty}}(h)$ & rate & $e_{l^{\infty}}(h)$ & rate& $e_{l^{\infty}}(h)$ &rate \\
\midrule
$2^{-4}$ & 2.34e-04 & * & 3.45e-03   & * & 3.88e-03 & * & 2.54e-02 & * \\
$2^{-5}$ & 2.41e-05  & 3.28 &   4.38e-04 & 2.98  & 6.06e-04 &  2.68&4.88e-03& 2.38\\
$2^{-6}$ & 2.46e-06 & 3.29 &  5.50e-05  & 2.99  & 9.37e-05 & 2.69  & 9.28e-04 &  2.39\\
$2^{-7}$ & 2.48e-07 & 3.31 &   6.84e-06 & 3.01 &  1.44e-05   & 2.70  & 1.76e-04  & 2.40\\
$2^{-8}$ & 2.42e-08 &  3.36  & 8.32e-07  & 3.04  &  2.19e-06  & 2.72 &  3.33e-05  & 2.40 \\
\bottomrule
\end{tabular}
\end{threeparttable}
\end{table}
\subsection{WIRFD simulations of the TFL equation}
This subsection investigates the performance of the WIRFD approximation for solving the TFL equation via \cref{alg:2}, focusing on the $l^2$-norm and $l^{\infty}$-norm errors, the associated convergence rates, and the efficiency of the preconditioner. The PCG solver in \cref{alg:2} employs a relative residual tolerance of $10^{-16}$, a maximum of $1500$ iterations, and a zero initial guess. For the exact grid function $U$ and its numerical approximation $U_h$, the $l^2$- and $l^{\infty}$-norm errors are measured by
\begin{align}\label{eq:05.3}
E_{l^2}(h)=\begin{cases}
\left\|U-U_h\right\|_{l^2(\Omega_h)},\quad &U \ \text{known},\\
\left\|U_h-U_{h/2}\right\|_{l^2(\Omega_h)},\quad &U\ \text{unknown},
\end{cases}\quad
E_{l^\infty}(h)=\begin{cases}
\left\|U-U_h\right\|_{l^{\infty}(\Omega_h)},\quad &U\ \text{known},\\
\left\|U_h-U_{h/2}\right\|_{l^{\infty}(\Omega_h)},\quad &U\ \text{unknown}.
\end{cases}
\end{align}
For the errors in \eqref{eq:05.3}, where $U_{h/2}$ is restricted to $\Omega_h$, the convergence rates are defined as $\log_{2}\left(\frac{E_{l^2}(2h)}{E_{l^2}(h)}\right)$ and $\log_{2}\left(\frac{E_{l^\infty}(2h)}{E_{l^\infty}(h)}\right)$.

To assess the performance and validate the error estimates from \cref{D4.2}, we consider the following problems.
\begin{example}\label{exam3}
For the $d=1$ TFL equation
\begin{alignat}{2}\label{eq:5.4}
(-\Delta)_{\lambda}^{\frac{\alpha}{2}}u(x) &= f(x), \quad && x \in (-1,1), \nonumber \\
u(x) &= 0, \quad && x \in \mathbb{R} \setminus (-1,1),
\end{alignat}
we take the exact solution as $u=(1-x^2)_+^p\, (p>0)$, which is given by \eqref{eq:5.1}.
\end{example}
In \cref{exam3}, an unknown source grid function $f$ is approximated by $f_{h_{\rm ref}}=(-\Delta_{h_{\rm ref}})_{\lambda}^{\frac{\alpha}{2}}U$ with $h_{\rm ref}=2^{-10}$. For $\lambda=0.5, p=8+\alpha$ with $\alpha=0.8, 1.6$, and $u\in C^{s}(\mathbb{R})$ where $s=\lfloor p\rfloor$, \cref{tab9} lists the $l^{\infty}$-norm error $E_{l^\infty}(h)$, the convergence rate, the iteration count (Iter), and the CPU time (in seconds) of the PCG and CG methods for solving the linear system with $h=2^{-6}, 2^{-7}, 2^{-8}, 2^{-9}$. For $\alpha=1.6$, the observed rate is consistent with \cref{D4.2}; the result for $\alpha=0.8$ provides numerical evidence beyond the proven range $\alpha\in[1,2)$. Moreover, PCG outperforms CG, requiring fewer iterations on all meshes and consuming less CPU time on finer meshes.

To further substantiate the error estimates in \cref{D4.2} and investigate the effects of $\lambda$ and $s$ on the convergence behavior, we perform numerical experiments with $\lambda=1.0$, $p=8.1$, and $u\in C^{8}(\mathbb{R})$. \cref{tab10,tab11} present the $l^2$- and $l^\infty$-norm errors, $E_{l^2}(h)$ and $E_{l^\infty}(h)$, together with the corresponding convergence rates of the numerical solution for $\alpha=0.7, 1.0, 1.3, 1.6$ and $h=2^{-5}, 2^{-6},\ldots, 2^{-9}$. Together with \cref{tab9}, these results show the rate $O(h^{4-\alpha})$ throughout the tested parameter range. The cases $\alpha\in[1,2)$ agree with \cref{D4.2}, whereas $\alpha=0.7$ again lies outside the current stability and convergence theory.

For $\lambda=1.0$ and $p=1+\alpha, 2+\alpha$ with $\alpha=0.8, 1.6$, \cref{tab12} gives
the $l^{\infty}$-norm error $E_{l^{\infty}}(h)$ and its convergence rate for $u\in C^{\lfloor p\rfloor}(\mathbb{R})$ with $h=2^{-5}, 2^{-6},\ldots, 2^{-9}$. The results suggest that the $l^{\infty}$-norm convergence rate remains $O(h^{4-\alpha})$ for moderately regular solutions and decreases approximately to $O(h^{\min\{p,4-\alpha\}})$ for less regular solutions. Thus, the method appears effective below the regularity threshold assumed in \cref{D4.2}, although this regime is not covered by the present analysis.

To clarify the role of $\lambda$ in the numerical error, \cref{fig:1} plots the $l^{\infty}$-norm error $E_{l^{\infty}}(h)$ for $p=8+\alpha$ with $\alpha=0.8, 1.6$ and $\lambda= 0.0, 0.5, 1.0, 1.5, 3.0, 5.0$ over spatial step sizes $h=2^{-4},2^{-5},\ldots, 2^{-9}$. For smaller $\alpha$, the error $E_{l^{\infty}}(h)$ increases with $\lambda$; for larger $\alpha$, it is nearly insensitive to $\lambda$. The observed rate remains $O(h^{4-\alpha})$ for all tested values of $\lambda$, showing that the order is retained across this parameter range.

\begin{table}[H]
\centering
\begin{threeparttable}
\caption{Error $E_{l^\infty}(h)$, convergence rate, iteration count (Iter), and CPU time of the PCG and CG solvers for \cref{exam3} ($d=1$, $\lambda=0.5$, $p=8+\alpha$).}
\label{tab9}

\setlength{\tabcolsep}{2.2pt}
\renewcommand{\arraystretch}{1.15}
\begin{tabular*}{0.95\textwidth}{
  @{\extracolsep{\fill}}
  c
  *{6}{c}
  @{\hspace{1.15em}}
  *{6}{c}
  @{}
}
\hline

$h$
& \multicolumn{6}{l}{$\alpha=0.8$}
& \multicolumn{6}{l}{$\alpha=1.6$} \\
\cmidrule(l{0.0em}r{1.0em}){2-7}
\cmidrule(l{0.0em}r{0.2em}){8-13}

& $E_{l^{\infty}}(h)$ & rate
& \multicolumn{2}{c}{PCG}
& \multicolumn{2}{c}{CG}
& $E_{l^{\infty}}(h)$ & rate
& \multicolumn{2}{c}{PCG}
& \multicolumn{2}{c}{CG} \\

\cmidrule(l{0.2em}r{0.2em}){4-5}
\cmidrule(l{0.2em}r{1.1em}){6-7}
\cmidrule(l{0.2em}r{0.2em}){10-11}
\cmidrule(l{0.2em}){12-13}

& & & Iter & CPU time & Iter & CPU time
& & & Iter & CPU time & Iter & CPU time \\
\hline
$2^{-6}$
&  2.94e-06  &*
& 12  & 6.50e-03
& 48 &  5.30e-03

& 8.54e-05 & *
& 17  &  6.85e-03
& 97  & 5.44e-03 \\

$2^{-7}$
& 3.21e-07 & 3.20
& 13  & 6.61e-03
& 66 &  5.99e-03

& 1.60e-05 & 2.41
& 21  & 6.91e-03
& 170   & 7.45e-03 \\

$2^{-8}$
& 3.49e-08 & 3.20
& 13  & 7.38e-03
& 89 &  8.75e-03

& 2.91e-06 & 2.46
& 23  &  8.43e-03
& 302  & 1.73e-02  \\

$2^{-9}$
& 3.53e-09 & 3.30
& 14   & 9.76e-03
& 120 & 1.32e-02

& 4.47e-07 &  2.70
& 27  &   9.22e-03
&  531  & 3.48e-02 \\
\hline
\end{tabular*}
\end{threeparttable}
\end{table}
\vspace{-0.6cm}
\begin{table}[H]
\centering
\begin{threeparttable}
\renewcommand{\arraystretch}{1.15} 
   \setlength{\tabcolsep}{9.pt}
\caption{Error $E_{l^{2}}(h)$ and convergence rate for \cref{exam3} ($d=1$, $\lambda=1.0$, $p=8.1$).}
\label{tab10}
\begin{tabular}{cccccccccc}
\toprule
\multirow{2}{*}{$h$} &
\multicolumn{2}{c}{$\alpha=0.7$} &
\multicolumn{2}{c}{$\alpha=1.0$} &
\multicolumn{2}{c}{$\alpha=1.3$}&
\multicolumn{2}{c}{$\alpha=1.6$}\\
\cmidrule(lr){2-3} \cmidrule(lr){4-5} \cmidrule(lr){6-7} \cmidrule(lr){8-9}
 & $E_{l^2}(h)$ & rate & $E_{l^2}(h)$ & rate & $E_{l^2}(h)$ & rate& $E_{l^2}(h)$ &rate \\
\midrule
$2^{-5}$ & 1.23e-05  & *  &   3.48e-05  & * & 9.41e-05  &  *& 2.43e-04 &*\\
$2^{-6}$ & 1.25e-06 & 3.29 &  4.37e-06  & 3.00 &  1.45e-05 & 2.70 & 4.61e-05 &  2.40\\
$2^{-7}$ & 1.28e-07 & 3.29 &   5.45e-07 & 3.00 &  2.22e-06   & 2.71   & 8.67e-06  & 2.41 \\
$2^{-8}$ & 1.32e-08 &  3.28  & 6.67e-08  &3.03&   3.29e-07  & 2.75 &   1.58e-06   & 2.46 \\
$2^{-9}$ & 1.29e-09 & 3.35 & 7.24e-09   &3.20  & 4.20e-08 & 2.97 &  2.45e-07  &  2.69 \\
\bottomrule
\end{tabular}
\end{threeparttable}
\end{table}
\vspace{-0.6cm}
\begin{table}[H]
\centering
\begin{threeparttable}
\renewcommand{\arraystretch}{1.15} 
   \setlength{\tabcolsep}{9.pt}
\caption{Error $E_{l^{\infty}}(h)$ and convergence rate for \cref{exam3} ($d=1$, $\lambda=1.0$, $p=8.1$).}
\label{tab11}
\begin{tabular}{cccccccccc}
\toprule
\multirow{2}{*}{$h$} &
\multicolumn{2}{c}{$\alpha=0.7$} &
\multicolumn{2}{c}{$\alpha=1.0$} &
\multicolumn{2}{c}{$\alpha=1.3$}&
\multicolumn{2}{c}{$\alpha=1.6$}\\
\cmidrule(lr){2-3} \cmidrule(lr){4-5} \cmidrule(lr){6-7} \cmidrule(lr){8-9}
 & $E_{l^{\infty}}(h)$ & rate & $E_{l^{\infty}}(h)$ & rate & $E_{l^{\infty}}(h)$ & rate& $E_{l^{\infty}}(h)$ &rate \\
\midrule
$2^{-5}$ & 1.96e-05 &  *  &   5.58e-05  & * & 1.52e-04   & *& 3.95e-04   &*\\
$2^{-6}$ & 2.00e-06  & 3.29  &  6.99e-06  &  3.00 &  2.34e-05  & 2.70 &  7.49e-05 &  2.40\\
$2^{-7}$ & 2.05e-07  & 3.29 &   8.71e-07& 3.01 &  3.56e-06    & 2.71   & 1.40e-05  & 2.41 \\
$2^{-8}$ & 2.12e-08 &  3.27  & 1.06e-07  & 3.04 & 5.23e-07  & 2.77 &   2.54e-06    & 2.47 \\
$2^{-9}$ & 2.13e-09 & 3.32 & 1.14e-08   &3.22  & 6.45e-08  & 3.02&  3.87e-07   &  2.71 \\
\bottomrule
\end{tabular}
\end{threeparttable}
\end{table}
\vspace{-0.6cm}
\begin{table}[H]
\centering
\begin{threeparttable}
\renewcommand{\arraystretch}{1.15} 
    \setlength{\tabcolsep}{9.pt}
\caption{Error $E_{l^{\infty}}(h)$ and convergence rate for \cref{exam3} ($d=1$, $\lambda=1.0$, $p=1+\alpha, 2+\alpha$).}
\label{tab12}
\begin{tabular}{cccccccccc}
\toprule
\multirow{2}{*}{$h$} &
\multicolumn{2}{c}{$p=1+\alpha,\ \alpha=0.8$} &
\multicolumn{2}{c}{$p=1+\alpha,\ \alpha=1.6$} &
\multicolumn{2}{c}{$p=2+\alpha,\ \alpha=0.8$}&
\multicolumn{2}{c}{$p=2+\alpha,\ \alpha=1.6$}\\
\cmidrule(lr){2-3} \cmidrule(lr){4-5} \cmidrule(lr){6-7} \cmidrule(lr){8-9} \cmidrule(lr){8-9}
 & $E_{l^{\infty}}(h)$ & rate & $E_{l^{\infty}}(h)$ & rate & $E_{l^{\infty}}(h)$ & rate & $E_{l^{\infty}}(h)$ & rate \\
\midrule
$2^{-5}$ &  5.09e-05    &  * &     1.29e-04  &  *    &      1.17e-05    &  *  &     1.43e-04  & *  \\
$2^{-6}$ & 1.30e-05 &  1.97   &    2.71e-05  & 2.25 &  1.40e-06 &  3.06 & 2.72e-05  &  2.40    \\
$2^{-7}$ & 3.50e-06 &  1.89  &    5.47e-06  & 2.31  &  1.64e-07   &    3.10 &   5.06e-06  &  2.42    \\
$2^{-8}$ &   9.67e-07   & 1.86    &   1.06e-06    & 2.37  &   1.87e-08    & 3.13    &  8.90e-07 &  2.51    \\
$2^{-9}$ &  2.56e-07  &1.92  &    1.74e-07  & 2.60 &   1.92e-09  &    3.28  &  1.41e-07  &  2.66   \\
\bottomrule
\end{tabular}
\end{threeparttable}
\end{table}
\vspace{-0.6cm}
\begin{figure}[H]
\centering
\includegraphics[width=0.45\linewidth]{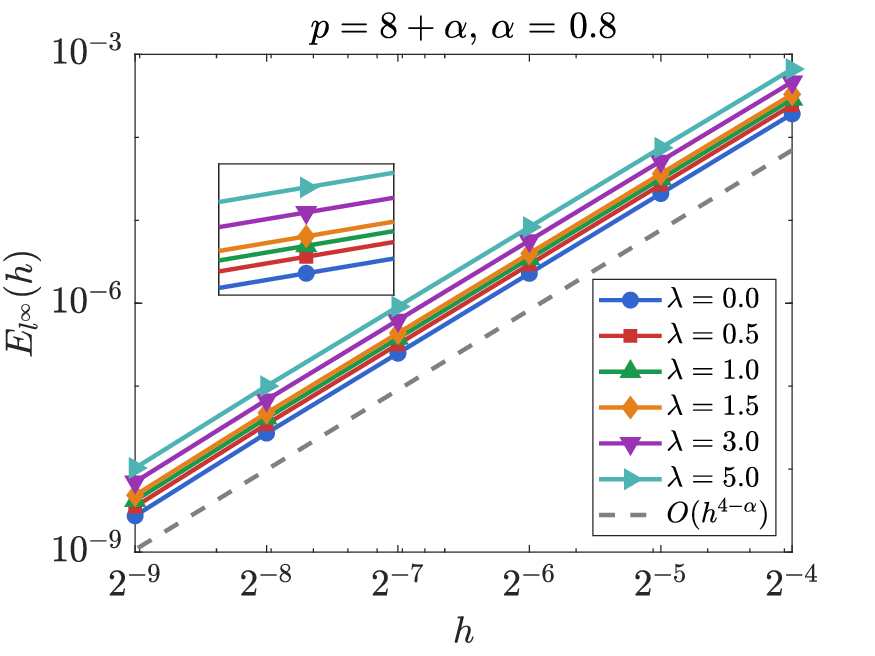}
\includegraphics[width=0.45\linewidth]{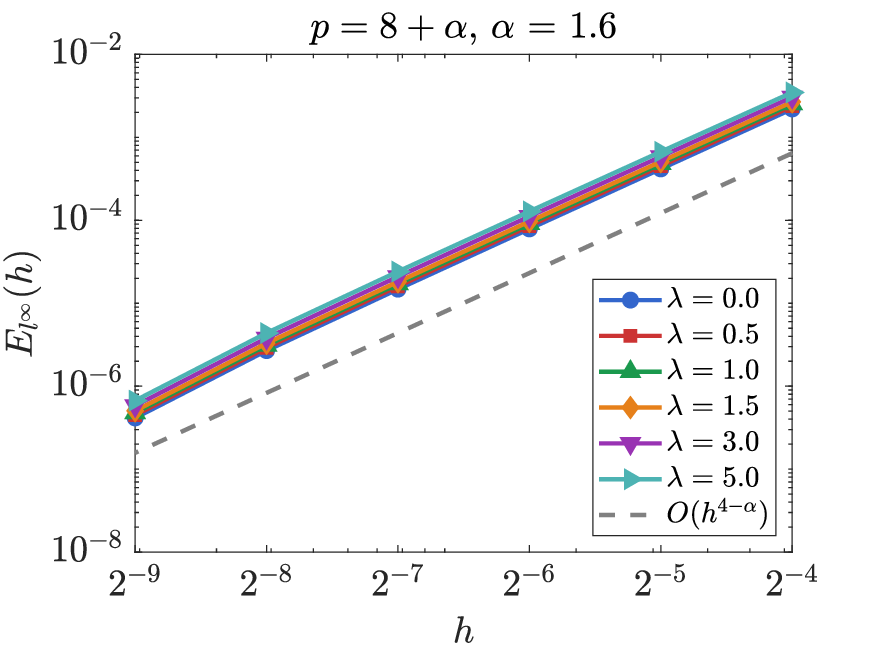}
\caption{Error $E_{l^\infty}(h)$ for the numerical solution of \cref{exam3}  ($\alpha = 0.8$, $1.6$; $\lambda = 0.0$, $0.5$, $1.0$, $1.5$, $3.0$, $5.0$).}
\label{fig:1}
\end{figure}
\allowdisplaybreaks
Although the analysis in \cref{D4.2} is confined to the case of $d=1$, the WIRFD method  is readily applicable to $d=2$ problems, as demonstrated in \cref{exam4}.
\begin{example}\label{exam4}
We consider the following $d=2$ TFL equation:
\begin{alignat}{2}\label{eq:5.3}
(-\Delta)_{\lambda}^{\frac{\alpha}{2}}u(x) &= f(x), \quad && x \in (-1,1)^2, \nonumber \\
u(x) &= 0, \quad && x \in\mathbb{R}^2 \setminus (-1,1)^2.
\end{alignat}
\end{example}
\textbf{Case 1.} For the exact solution  $u(x)=\left[(1-x_1^2)_+(1-x_2^2)_+\right]^p$ with $p>0$ and the unknown $f$, we take $f\approx f_{h_{\rm ref}}=(-\Delta_{h_{\rm ref}})_{\lambda}^{\frac{\alpha}{2}}U$ with $h_{\rm ref}=2^{-10}$ and solve the resulting system. Adopting the same setup as in \cref{tab9,tab10,tab11,tab12} for \cref{exam3}, the corresponding numerical results for the $d=2$ case are reported in \cref{tab13,tab14,tab15,tab16}. For smooth solutions, the observed rate is $O(h^{4-\alpha})$ across the tested values of $\lambda$. The results also suggest a rate of approximately $O(h^{\min\{p,4-\alpha\}})$ for less regular solutions. These findings demonstrate the practical effectiveness of the method in a regime not covered by the present higher-dimensional solution error analysis. Notably, PCG exhibits a clear advantage over CG for fine $h$, whereas the two solvers perform comparably for small $\alpha$ and coarse $h$.

\textbf{Case 2.} For the source function $f(x)=1$, whose exact solution is unknown, we evaluate the $l^2$- and $l^\infty$-norm errors defined in \eqref{eq:05.3}. \cref{tab17} presents the $l^{\infty}$-norm error $E_{l^\infty}(h)$, the corresponding convergence rate, the iteration count (Iter), and the CPU time (in seconds) for the PCG and CG solvers, with $\lambda=0.5$, $\alpha=0.8, 1.6$, and $h=2^{-5}, 2^{-6}, \ldots, 2^{-9}$. The error $E_{l^\infty}(h)$ decays at an observed convergence rate of $O(h^{\frac{\alpha}{2}})$, and PCG consistently outperforms CG in both iteration count and CPU time. Furthermore, \cref{tab18} reports the $l^2$-norm error $E_{l^2}(h)$ for $\lambda=0.5$, $\alpha=0.8, 1.0, 1.3, 1.6$, on the same grid, where the observed convergence rate agrees with $O\left(h^{\min\left\{1,\frac{1+\alpha}{2}\right\}}\right)$.
\cref{tab19} summarizes the $l^{\infty}$-norm error $E_{l^{\infty}}(h)$ and convergence rate for $\lambda=0.2$, $\alpha=0.8, 1.0, 1.3, 1.6$, and $h=2^{-5}, 2^{-6}, \ldots, 2^{-9}$, where the $l^{\infty}$-norm convergence rate is again $O(h^{\frac{\alpha}{2}})$. Comparison of \cref{tab17} ($\lambda=0.5$) and \cref{tab19} ($\lambda=0.2$) for $\alpha=0.8, 1.6$ shows that $E_{l^{\infty}}(h)$ is larger for larger $\lambda$ when $\alpha$ is small but is nearly insensitive to $\lambda$ for larger $\alpha$, while the observed convergence rate remains $O(h^{\frac{\alpha}{2}})$ for both tested values of $\lambda$.
\vspace{-0.2cm}
\begin{table}[H]
\centering
\begin{threeparttable}
\caption{Error $E_{l^\infty}(h)$, convergence rate, iteration count (Iter), and CPU time of the PCG and CG solvers for $\textbf{Case 1}$ ($d=2$, $\lambda=0.5$, $p=8+\alpha$).}
\label{tab13}

\setlength{\tabcolsep}{2.2pt}
\renewcommand{\arraystretch}{1.15}
\begin{tabular*}{0.95\textwidth}{
  @{\extracolsep{\fill}}
  c
  *{6}{c}
  @{\hspace{1.15em}}
  *{6}{c}
  @{}
}
\hline

$h$
& \multicolumn{6}{l}{$\alpha=0.8$}
& \multicolumn{6}{l}{$\alpha=1.6$} \\
\cmidrule(l{0.0em}r{1.0em}){2-7}
\cmidrule(l{0.0em}r{0.2em}){8-13}

& $E_{{l^\infty}}(h)$ & rate
& \multicolumn{2}{c}{PCG}
& \multicolumn{2}{c}{CG}
& $E_{{l^\infty}}(h)$ & rate
& \multicolumn{2}{c}{PCG}
& \multicolumn{2}{c}{CG} \\

\cmidrule(l{0.2em}r{0.2em}){4-5}
\cmidrule(l{0.2em}r{1.1em}){6-7}
\cmidrule(l{0.2em}r{0.2em}){10-11}
\cmidrule(l{0.2em}){12-13}

& & & Iter & CPU time & Iter & CPU time
& & & Iter & CPU time & Iter & CPU time \\
\hline
$2^{-6}$
& 4.95e-06  & *
& 24  & 4.69e-01
&  71  &  4.17e-01

& 1.31e-04 & *
& 45 & 5.26e-01
&  229  & 8.39e-01\\

$2^{-7}$
& 5.31e-07 & 3.22
& 27  & 1.37e+00
& 95  & 1.65e+00

&  2.47e-05 & 2.41
&  62   & 1.68e+00
& 400    &  3.83e+00 \\

$2^{-8}$
&  5.40e-08 & 3.30
& 30  &  5.35e+00
& 127 &  8.33e+00

& 4.53e-06 & 2.45
& 89  &  8.48e+00
& 752 &  2.73e+01  \\

$2^{-9}$
& 4.40e-09 & 3.62
& 33  &  2.13e+01
& 169 & 7.74e+01

& 7.10e-07 &  2.67
& 127  &  3.53e+01
&  1375   &  1.75e+02   \\
\hline
\end{tabular*}
\end{threeparttable}
\end{table}
\vspace{-0.9cm}
\begin{table}[H]
\centering
\begin{threeparttable}
\renewcommand{\arraystretch}{1.15} 
   \setlength{\tabcolsep}{9.pt}
\caption{Error $E_{l^{2}}(h)$ and convergence rate for $\textbf{Case 1}$ ($d=2$, $\lambda=1.0$, $p=8.1$).}
\label{tab14}
\begin{tabular}{cccccccccc}
\toprule
\multirow{2}{*}{$h$} &
\multicolumn{2}{c}{$\alpha=0.7$} &
\multicolumn{2}{c}{$\alpha=1.0$} &
\multicolumn{2}{c}{$\alpha=1.3$}&
\multicolumn{2}{c}{$\alpha=1.6$}\\
\cmidrule(lr){2-3} \cmidrule(lr){4-5} \cmidrule(lr){6-7} \cmidrule(lr){8-9}
 & $E_{l^2}(h)$ & rate & $E_{l^2}(h)$ & rate & $E_{l^2}(h)$ & rate& $E_{l^2}(h)$ &rate \\
\midrule
$2^{-5}$ & 1.15e-05  &  *  &   3.12e-05  & * & 8.15e-05  &  *& 2.05e-04& *\\
$2^{-6}$ & 1.17e-06  & 3.30  &  3.91e-06  & 3.00 &  1.26e-05 & 2.70 & 3.89e-05 &  2.40\\
$2^{-7}$ & 1.17e-07  & 3.32 &   4.83e-07 & 3.02  &  1.92e-06  & 2.71   & 7.32e-06 & 2.41 \\
$2^{-8}$ & 1.12e-08 & 3.39 & 5.72e-08  &3.08&  2.83e-07  & 2.76 & 1.34e-06 & 2.45 \\
$2^{-9}$ & 8.80e-10 & 3.67 & 5.66e-09 & 3.34 & 3.56e-08  & 2.99 &  2.10e-07  &  2.67 \\
\bottomrule
\end{tabular}
\end{threeparttable}
\end{table}
\vspace{-0.6cm}
\begin{table}[H]
\centering
\begin{threeparttable}
\renewcommand{\arraystretch}{1.15} 
   \setlength{\tabcolsep}{9.pt}
\caption{Error $E_{l^{\infty}}(h)$ and convergence rate for $\textbf{Case 1}$ ($d=2$, $\lambda=1.0$, $p=8.1$).}
\label{tab15}
\begin{tabular}{cccccccccc}
\toprule
\multirow{2}{*}{$h$} &
\multicolumn{2}{c}{$\alpha=0.7$} &
\multicolumn{2}{c}{$\alpha=1.0$} &
\multicolumn{2}{c}{$\alpha=1.3$}&
\multicolumn{2}{c}{$\alpha=1.6$}\\
\cmidrule(lr){2-3} \cmidrule(lr){4-5} \cmidrule(lr){6-7} \cmidrule(lr){8-9}
 & $E_{{l^\infty}}(h)$ & rate & $E_{{l^\infty}}(h)$ & rate & $E_{{l^\infty}}(h)$ & rate& $E_{{l^\infty}}(h)$ &rate \\
\midrule
$2^{-5}$ & 3.44e-05 & *  &   9.30e-05 & * & 2.42e-04 &  * &  6.06e-04  & *\\
$2^{-6}$ & 3.49e-06  & 3.30  &  1.16e-05 &  3.00 & 3.73e-05  & 2.70 &1.15e-04&  2.40\\
$2^{-7}$ & 3.49e-07  & 3.33 &   1.43e-06 & 3.02 & 5.69e-06  & 2.71   & 2.17e-05 & 2.41 \\
$2^{-8}$ & 3.27e-08  & 3.41 & 1.68e-07  & 3.09 & 8.35e-07  & 2.77 & 3.96e-06  & 2.45 \\
$2^{-9}$ & 2.37e-09 &  3.79  & 1.58e-08 & 3.41 &  1.03e-07  & 3.02& 6.19e-07  &  2.68 \\
\bottomrule
\end{tabular}
\end{threeparttable}
\end{table}
\vspace{-0.6cm}
\begin{table}[H]
\centering
\begin{threeparttable}
\renewcommand{\arraystretch}{1.2} 
    \setlength{\tabcolsep}{9.pt}
\caption{Error $E_{l^{\infty}}(h)$ and convergence rate for $\textbf{Case 1}$ ($d=2$, $\lambda=1.0$, $p=1+\alpha, 2+\alpha$).}
\label{tab16}
\begin{tabular}{cccccccccc}
\toprule
\multirow{2}{*}{$h$} &
\multicolumn{2}{c}{$p=1+\alpha,\ \alpha=0.8$} &
\multicolumn{2}{c}{$p=1+\alpha,\ \alpha=1.6$} &
\multicolumn{2}{c}{$p=2+\alpha,\ \alpha=0.8$}&
\multicolumn{2}{c}{$p=2+\alpha,\ \alpha=1.6$}\\
\cmidrule(lr){2-3} \cmidrule(lr){4-5} \cmidrule(lr){6-7} \cmidrule(lr){8-9} \cmidrule(lr){8-9}
 & $E_{l^\infty}(h)$ & rate & $E_{l^\infty}(h)$ & rate & $E_{l^\infty}(h)$ & rate & $E_{l^\infty}(h)$ & rate \\
\midrule
$2^{-5}$ &  5.25e-05    & * &     1.45e-04  &  *   &      1.06e-05    &  *  &     2.21e-04  & *  \\
$2^{-6}$ &  1.33e-05 &  1.98   &    2.60e-05 & 2.48 &  1.27e-06  &  3.05 & 4.19e-05 &  2.40    \\
$2^{-7}$ & 3.57e-06 &  1.90  &    4.66e-06  &  2.48  &  1.49e-07  &    3.09 &   7.87e-06  &  2.41   \\
$2^{-8}$ &  9.80e-07   & 1.87    &   9.02e-07   & 2.37  &  1.70e-08   & 3.14  &  1.42e-06 &  2.47    \\
$2^{-9}$ &  2.58e-07   &1.92  &    1.49e-07   & 2.60 & 1.73e-09 &    3.29 &  2.16e-07  &   2.72   \\
\bottomrule
\end{tabular}
\end{threeparttable}
\end{table}
\vspace{-0.6cm}
\begin{table}[H]
\centering
\begin{threeparttable}
\caption{Error $E_{l^\infty}(h)$, convergence rate, iteration count (Iter), and CPU time of the PCG and CG solvers for $\textbf{Case 2}$ ($d=2$, $\lambda=0.5$).}
\label{tab17}

\setlength{\tabcolsep}{2.2pt}
\renewcommand{\arraystretch}{1.15}
\begin{tabular*}{0.95\textwidth}{
  @{\extracolsep{\fill}}
  c
  *{6}{c}
  @{\hspace{1.15em}}
  *{6}{c}
  @{}
}
\hline

$h$
& \multicolumn{6}{l}{$\alpha=0.8$}
& \multicolumn{6}{l}{$\alpha=1.6$} \\
\cmidrule(l{0.0em}r{1.0em}){2-7}
\cmidrule(l{0.0em}r{0.2em}){8-13}

& $E_{l^{\infty}}(h)$ & rate
& \multicolumn{2}{c}{PCG}
& \multicolumn{2}{c}{CG}
& $E_{l^{\infty}}(h)$ & rate
& \multicolumn{2}{c}{PCG}
& \multicolumn{2}{c}{CG} \\

\cmidrule(l{0.2em}r{0.2em}){4-5}
\cmidrule(l{0.2em}r{1.1em}){6-7}
\cmidrule(l{0.2em}r{0.2em}){10-11}
\cmidrule(l{0.2em}){12-13}

& & & Iter & CPU time & Iter & CPU time
& & & Iter & CPU time & Iter & CPU time \\
\hline
%
%
$2^{-5}$
& 1.63e-01 &  *
& 22 & 1.26e-01
& 56 & 1.40e-01

& 5.58e-03 &  *
& 34  &1.28e-01
& 136  &   1.60e-01   \\
$2^{-6}$
& 1.18e-01  & 0.46
& 25  &  3.85e-01
&  76  &  4.57e-01

& 3.19e-03 & 0.80
&  48 & 4.46e-01
&  236 & 7.65e-01\\

$2^{-7}$
& 8.66e-02 & 0.44
& 28  & 1.31e+00
& 102  &1.84e+00

&  1.83e-03 & 0.80
&   70   & 1.64e+00
&  440    & 4.09e+00    \\

$2^{-8}$
& 6.43e-02  & 0.43
& 32  &    5.36e+00
& 137 &   1.04e+01

& 1.05e-03 & 0.80
& 96  &   7.58e+00
& 723 &   5.13e+01   \\

$2^{-9}$
& 4.80e-02 & 0.42
& 35  &   2.11e+01
& 182 &  5.49e+01

& 6.00e-04 &  0.80
& 140   &  3.46e+01
&  1409   &  2.19e+02   \\
\hline
\end{tabular*}
\end{threeparttable}
\end{table}
\vspace{-0.6cm}
\begin{table}[H]
\centering
\begin{threeparttable}
\renewcommand{\arraystretch}{1.15} 
   \setlength{\tabcolsep}{9.pt}
\caption{Error $E_{l^{2}}(h)$ and convergence rate for $\textbf{Case 2}$ ($d=2$, $\lambda=0.5$).}
\label{tab18}
\begin{tabular}{cccccccccc}
\toprule
\multirow{2}{*}{$h$} &
\multicolumn{2}{c}{$\alpha=0.8$} &
\multicolumn{2}{c}{$\alpha=1.0$} &
\multicolumn{2}{c}{$\alpha=1.3$}&
\multicolumn{2}{c}{$\alpha=1.6$}\\
\cmidrule(lr){2-3} \cmidrule(lr){4-5} \cmidrule(lr){6-7} \cmidrule(lr){8-9}
 & $E_{l^2}(h)$ & rate & $E_{l^2}(h)$ & rate & $E_{l^2}(h)$ & rate& $E_{l^2}(h)$ &rate \\
\midrule
$2^{-5}$ & 1.05e-01 &  *  &  1.51e-02  & * & 3.00e-02 & * & 7.89e-03 & *\\
$2^{-6}$ & 5.81e-02 & 0.85  &  8.16e-03  & 0.89 & 1.58e-02 & 0.93 & 4.07e-03 &  0.95\\
$2^{-7}$ & 3.18e-02  & 0.87 & 4.35e-03 & 0.91  & 8.19e-03  & 0.95   & 2.08e-03 & 0.97 \\
$2^{-8}$ & 1.73e-02 & 0.88 & 2.30e-03  & 0.92&  4.21e-03  & 0.96 & 1.05e-03 & 0.98 \\
$2^{-9}$ & 9.35e-03 & 0.89 &  1.21e-03 &  0.93 & 2.15e-03  & 0.97 &  5.32e-04  &  0.99 \\
\bottomrule
\end{tabular}
\end{threeparttable}
\end{table}
\vspace{-0.3cm}
\begin{table}[H]
\centering
\begin{threeparttable}
\renewcommand{\arraystretch}{1.15} 
   \setlength{\tabcolsep}{9.pt}
\caption{Error $E_{l^{\infty}}(h)$ and convergence rate for $\textbf{Case 2}$ ($d=2$, $\lambda=0.2$).}
\label{tab19}
\begin{tabular}{cccccccccc}
\toprule
\multirow{2}{*}{$h$} &
\multicolumn{2}{c}{$\alpha=0.8$} &
\multicolumn{2}{c}{$\alpha=1.0$} &
\multicolumn{2}{c}{$\alpha=1.3$}&
\multicolumn{2}{c}{$\alpha=1.6$}\\
\cmidrule(lr){2-3} \cmidrule(lr){4-5} \cmidrule(lr){6-7} \cmidrule(lr){8-9}
 & $E_{l^{\infty}}(h)$ & rate & $E_{l^{\infty}}(h)$ & rate & $E_{l^{\infty}}(h)$& rate& $E_{l^{\infty}}(h)$ &rate \\
\midrule
$2^{-5}$ &  1.25e-01 &   *  &  1.60e-02 & * &  2.56e-02 &  *&  5.16e-03   & *\\
$2^{-6}$ & 9.23e-02  & 0.44  &  1.12e-02 &  0.52  & 1.62e-02 &  0.66  &2.96e-03&  0.80\\
$2^{-7}$ & 6.87e-02  & 0.42 &   7.81e-03 & 0.51 & 1.03e-02  & 0.66  & 1.70e-03 & 0.80 \\
$2^{-8}$ & 5.15e-02  & 0.42 &  5.49e-03  & 0.51 & 6.55e-03  & 0.65 & 9.76e-04  & 0.80 \\
$2^{-9}$ & 3.87e-02 & 0.41 & 3.86e-03 & 0.51 &  4.17e-03 & 0.65& 5.60e-04  &  0.80 \\
\bottomrule
\end{tabular}
\end{threeparttable}
\end{table}
\section{Conclusion}
In this work, we developed a WIRFD scheme for the TFL. For $d=1,2,3$ and $\alpha\in(0,2)$, we established an $O(h^{4-\alpha})$ truncation error bound in the $l^\infty$-norm under the regularity assumption $u\in C^s(\mathbb{R}^d)$ with $s\geq8$. For the one-dimensional TFL equation with $\alpha\in[1,2)$, we further proved $l^2$- and $l^\infty$-stability and the corresponding $O(h^{4-\alpha})$ convergence of the numerical solution by exploiting strict diagonal dominance and a lower bound for the smallest eigenvalue. Numerical experiments support these theoretical results.

A rigorous theoretical analysis for the higher-dimensional equation, however, remains a substantial challenge due to the inherent complexity of the computational stencil coefficients. Nonetheless, numerical experiments for the $d=2$ TFL equation with $u\in C^s(\mathbb{R}^2)$, $s\geq8$, exhibit an $O(h^{4-\alpha})$ rate in both the $l^2$- and $l^{\infty}$-norm errors, providing supporting evidence for the effectiveness of the scheme in higher dimensions. Moreover, the resulting multilevel Toeplitz stiffness matrix not only facilitates fast matrix-vector multiplications but also admits efficient preconditioning strategies, enhancing overall computational efficiency.

A particularly noteworthy observation concerns the source function $f(x)=1$, for which the measured convergence deteriorates to  $O\left(h^{\min\left\{1,\frac{1+\alpha}{2}\right\}}\right)$ and $O(h^{\frac{\alpha}{2}})$ in the $l^2$- and $l^{\infty}$-norms, respectively. This phenomenon is rooted in the well-established regularity theory for $\lambda=0$, wherein the solution exhibits only H\"{o}lder continuity of $\frac{\alpha}{2}$ order in the vicinity of the boundary due to the interplay of the fractional Laplacian with the Dirichlet boundary condition. Consequently, the convergence rate of the $l^{\infty}$-norm error cannot be expected to surpass $O(h^{\frac{\alpha}{2}})$. The experiments suggest that an analogous regularity limitation also persists when $\lambda\neq0$.

Overall, the proposed WIRFD scheme provides a reliable and efficient numerical framework for the TFL equation. Bridging the gap between these numerical observations and the corresponding theoretical analysis, especially for higher-dimensional problems and low-regularity solutions, remains the subject of our ongoing investigations.
\bibliographystyle{spmpsci}
\bibliography{TWFD_JSC}
\end{document}